\RequirePackage{fix-cm}
\documentclass[smallextended]{svjour3}       
\smartqed  
\usepackage{graphicx}
\usepackage{amsmath}
\usepackage{amssymb}
\usepackage{cite}
\usepackage{xcolor}

\usepackage{tikz}
\usetikzlibrary{decorations.markings}
\usetikzlibrary{decorations.pathmorphing}
\usetikzlibrary{arrows, snakes}
\usetikzlibrary{shapes.geometric, calc,decorations.pathreplacing}
\usetikzlibrary{backgrounds}
\usetikzlibrary{fit}
\usetikzlibrary{decorations.pathreplacing}
\input xy
\xyoption{all}

\tikzset{middlearrow/.style={
		decoration={markings,
			mark= at position 0.5 with {\arrow{#1}} ,
		},
		postaction={decorate}
	}
}

\DeclareMathOperator*{\Hess}{Hess}

\DeclareMathOperator*{\argmin}{arg\,min}

\journalname{Information Geometry}

\begin{document}

\title{Pseudo-Riemannian geometry embeds information geometry in optimal transport\thanks{The research of T.-K.~L.~Wong is partially supported by NSERC Grant RGPIN-2019-04419 and the Connaught New Researcher Award. J.~Yang  was partially supported by a Simons Investigator Award and NSF grant CCF-1815254.}
}

\titlerunning{Embedding information geometry in optimal transport}        

\author{Ting-Kam Leonard Wong         \and
        Jiaowen Yang 
}


\institute{T.-K.~L.~Wong \at
              Department of Statistical Sciences, University of Toronto \\
              \email{tkl.wong@utoronto.ca}           
           \and
           J.~Yang \at Facebook \\
           \email{jiaowen@fb.com}
}

\date{Received: date / Accepted: date}

\maketitle

\begin{abstract}
Optimal transport and information geometry both study geometric structures on spaces of probability distributions. Optimal transport characterizes the cost-minimizing movement from one distribution to another, while information geometry originates from coordinate-invariant properties of statistical inference. Their connections and applications in statistics and machine learning have started to gain more attention. In this paper we give a new differential geometric connection between the two fields. Namely, the pseudo-Riemannian framework of Kim and McCann, a geometric perspective on the fundamental Ma-Trudinger-Wang (MTW) condition in the regularity theory of optimal transport maps, encodes the dualistic structure of statistical manifold. This general relation is described using the natural framework of $c$-divergence, a divergence defined by an optimal transport map. As a by-product, we obtain a new information-geometric interpretation of the MTW tensor. This connection sheds light on old and new aspects of information geometry. The dually flat geometry of Bregman divergence corresponds to the quadratic cost and the pseudo-Euclidean space, and the $L^{(\alpha)}$-divergence introduced by Pal and the first author has constant sectional curvature in a sense to be made precise. In these cases we give a geometric interpretation of the information-geometric curvature in terms of the divergence between a primal-dual pair of geodesics.
\end{abstract}

\section{Introduction}

Let $\mu$ and $\nu$ be Borel probability measures on Polish spaces $M$ and $M'$ respectively. Given a real-valued cost function $c$ defined on $M \times M'$, the Monge-Kantorovich optimal transport problem is
\begin{equation} \label{eqn:MK.problem}
	\mathcal{T}_c(\mu, \nu) := \inf_{\gamma \in \Pi(\mu, \nu)} \int_{M \times M'} c d \gamma,
\end{equation}
where $\Pi(\mu, \nu)$ is the set of probability measures on $M \times M'$ whose first and second marginals are $\mu$ and $\nu$ respectively. For systematic mathematical expositions of optimal transport theory we refer the reader to the textbooks \cite{V03, V08, S15}. Under suitable conditions on the cost function $c$ and the measures $\mu, \nu$, the optimal coupling $\gamma^*$ can be shown to be deterministic, i.e., $\gamma^*$ concentrates on the graph $G \subset M \times M'$ of a measurable map $f: M \rightarrow M'$, called the optimal transport map. The transport cost $\mathcal{T}_c$ may be regarded as a lift of the cost $c$ from $M \times M'$ to $\mathcal{P}(M) \times \mathcal{P}(M')$, where $\mathcal{P}(M)$ and $\mathcal{P}(M')$ are spaces of probability measures on $M$ and $M'$ respectively. For example, when $M = M'$ and $c(x, y) = d(x, y)^p$, where $d$ is a metric and $p \geq 1$, then $\mathcal{T}_c^{1/p}$ is the Wasserestein metric of order $p$. Optimal transport has deep and elegant connections with probability, geometry and analysis. Thanks to recent breakthroughs in algorithmic development, optimal transport was also shown to be remarkably useful in statistical and machine learning applications \cite{PC19}.

In this paper we let $M$ and $M'$ be $n$-dimensional smooth manifolds ($n \geq 2$). In many cases of interest these are open domains in Euclidean space. Also we assume that $c: M \times M' \rightarrow \mathbb{R}$ is smooth. Following the geometric approach of Kim and McCann \cite{KM10} and McCann \cite{M12}, consider the {\it cross-difference}
\begin{equation} \label{eqn:cross.difference}
	\delta(p, q', p_0, q_0') := c(p, q_0') + c(p_0, q') - c(p, q') - c(p_0, q_0')
\end{equation}
defined for pairs $(p, q'), (p_0, q_0')$ of the product space $M \times M'$. Intuitively, it compares the costs of the matching $(p \rightarrow q', p_0 \rightarrow q_0')$ with that of $(p \rightarrow q_0', p_0 \rightarrow q')$. Note that if both $(p, q')$ and $(p_0, q_0')$ belong to the graph $G$ of an optimal transport map (for some pair $(\mu, \nu)$), by the $c$-cyclical monotonicity of $G$ we have $\delta(p, q', p_0, q_0') \geq 0$. In \cite{KM10, M12} it was shown that this cross-difference defines a pseudo-Riemannian geometry on $M \times M'$ which controls the geometry -- including the regularity -- of the optimal coupling. The pseudo-Riemannian metric (a non-degenerate $2$-tensor), given by $h = \frac{1}{2} \Hess \delta$, has signature $(n, n)$ where the Hessian is taken with respect to the $2n$-dimensional variable $(p, q')$. The details of this construction are recalled in Section \ref{sec:pseudo.Riemannian}. In this paper we show that the pseudo-Riemannian framework also encodes the dualistic structure of statistical manifold, and use this connection to elucidate several aspects of information geometry. 

Before describing the connection with information geometry, let us discuss briefly the context of the papers \cite{KM10, M12} in more detail. A fundamental problem in optimal transport is to study when the Monge-Kantorovich problem \eqref{eqn:MK.problem} admits a Monge (i.e., deterministic) solution, and, if so, the regularity (continuity and smoothness) of the optimal transport map. While existence of a Monge solution is implied by a twist condition on the cost function, regularity involves analyzing nonlinear partial differential equations of which the Monge-Amp\`{e}re equation is a classic example. Significant progress was achieved by Ma, Trudinger and Wang \cite{MTW05} and Trudinger and Wang \cite{TW09} who introduced a fourth order differential condition on the cost $c$ and showed that it is sufficient for continuity and regularity estimates of the transport map (under suitable conditions on $\mu$ and $\nu$  and their supports). Loeper \cite{L09} showed that this condition is also necessary and also observed that for a Riemannian manifold $M$ with $c(x, y) = \frac{1}{2} d^2(x, y)$, the Ma-Trudinger-Wang (MTW) tensor on the diagonal of $M \times M$ is proportional to the Riemannian sectional curvature (see Example \ref{ex:Loeper}). Kim and McCann \cite{KM10} showed that these results correspond to conditions on the Riemann curvature tensor in their pseudo-Riemannian framework (see Remark \ref{rmk:MTW}). Further progress, including a characterization of the graph $G$ in terms of a calibration form \cite{KMW10}, is surveyed in \cite{M12}. For more details about the regularity of optimal transport maps we refer the reader to \cite[Chapter 12]{V08}, \cite[Section 5]{KM10} as well as \cite{GA14}.

\subsection{Summary of main results} \label{sec:summary}
To describe the connection between the pseudo-Riemannian framework and information geometry, we first recall the concept of divergence which is a fundamental concept in information geometry \cite{A16, AJVS17}. For simplicity we assume that the objects considered (manifolds, functions, etc) are all smooth.

\begin{definition}[Divergence] \label{def:divergence}
	Let $M$ be an $n$-dimensional manifold. A divergence on $M$ is a function $\mathbf{D} : M \times M \rightarrow [0, \infty)$ such that the following properties hold:
	\begin{enumerate}
		\item[(i)] ${\bf D}[ p : p'] = 0$ only if $p = p'$.
		\item[(ii)] In local coordinates $(\xi^1, \ldots, \xi^n)$, we have
		\begin{equation} \label{eqn:divergence.local.quadratic}
			{\bf D}[ \xi + \Delta \xi : \xi] = \frac{1}{2} g_{ij}(\xi) \Delta \xi^i \Delta \xi^j + O(|\Delta \xi|^3),
		\end{equation}
		where $(g_{ij}(\xi))$ is strictly positive definite and varies smoothly in $\xi$. Note that $g$ defines a Riemannian metric on $M$.
	\end{enumerate}
\end{definition}

In \eqref{eqn:divergence.local.quadratic} and throughout the paper the Einstein summation convention is used. The classic example is where $M$ is a finite-dimensional family $\{p(\cdot ; \xi)\}$ of parameterized probability densities (such as an exponential family, see \cite[Chapter 2]{A16}), and ${\bf D}$ is the Kullback-Leibler divergence (relative entropy). Then the corresponding Riemannian metric is the Fisher information metric given by
\[
g_{ij}(\xi) = \int \frac{\partial \log p}{\partial \xi^i}(x ; \xi) \frac{\partial \log p}{\partial \xi^j}(x ; \xi) p(x ; \xi) dx.
\]
Following Eguchi \cite{E83}, a divergence induces a dualistic structure $(g, \nabla, \nabla^*)$ on $M$ consisting of a Riemannian metric $g$, as in \eqref{eqn:divergence.local.quadratic}, and a pair $(\nabla, \nabla^*)$ of torsion-free affine connections which are dual with respect to the metric $g$ (the expressions are given in Section \ref{sec:coeff.prelim}). The connections are defined in terms of third order derivatives of ${\bf D}[ \cdot : \cdot]$ on the diagonal (which may be regarded as the graph $G = \{ (x, x): x \in M\}$ of the identity map $f(x) = x$, see Example \ref{eg:any.divergence} below), and the duality means that for any vector fields $X$, $Y$ and $Z$ we have the following extension of the metric property:
\begin{equation*} 
	Z g( X, Y) = g( \nabla_Z X, Y ) + g( X, \nabla_Z^* Y ).
\end{equation*}
Note that the average $\frac{1}{2} (\nabla + \nabla^*)$ of $\nabla$ and $\nabla^*$ is the Levi-Civita connection of $g$.\footnote{Equivalently, a dualistic structure (also called a statistical manifold) can be defined by $(M, g, T)$, where $T = \nabla - \nabla^*$ can be interpreted as a symmetric $3$-tensor \cite{L87}. It can be shown that for any dualistic structure there exists a divergence which induces it.} When ${\bf D}$ is a Bregman divergence the induced dualistic structure is dually flat, i.e., both $\nabla$ and $\nabla^*$ are flat \cite{NA82}. Moreover, the two affine coordinate systems are related by a Legendre transformation \cite[Chapter 6]{A16}. In the context of exponential family, this reduces to the convex duality between the canonical and expectation parameters. Dual connections also appear in the context of affine differential geometry \cite{NS94}. There are other geometric structures in information geometry, including statistical manifolds admitting torsion \cite{HM19}, quantum information geometry \cite{ZGC07} as well as infinite dimensional statistical manifolds \cite{PS95}, but in this paper we focus on the dualistic structure $(g, \nabla, \nabla^*)$.

In Section \ref{sec:c.divergence} we review the framework of $c$-divergence, a divergence on the graph $G \subset M \times M'$ of optimal transport derived from the optimal transport map $f$ and the corresponding Kantorovich potentials. Introduced by Pal and the first author in \cite{PW18, W18}, this framework includes the Bregman divergence as well as the $L^{(\alpha)}$-divergence studied in \cite{W15, PW16, PW18, W18, W19}. The $L^{(\alpha)}$-divergence has many interesting properties; here we only note that its dualistic structure is dually projectively flat with constant sectional curvature $-\alpha$ and leads naturally to a generalized exponential families \cite{W18}. In fact, {\it  any} divergence in the sense of Definition \ref{def:divergence} can be regarded as a $c$-divergence by letting $c = {\bf D}$ (see Example \ref{eg:any.divergence}). Given the optimal transport map $f$, we regard the graph $G = \{(p, f(p)) : p \in M\}$ as an embedded submanifold of the product manifold $M \times M'$, and the $c$-divergence as a divergence on $G$. This gives a dualistric structure $(g, \nabla, \nabla^*)$ on $G$. 

In Section \ref{sec:geometry.connection} we study the relations between the two geometries $(M \times M', h)$ and $(G, g, \nabla, \nabla^*)$. Our main result, informally stated, reads as follows.

\begin{theorem} \label{thm:sec3}
	The information-geometric Riemannian metric $g$ is the restriction of the pseudo-Riemannian metric $h$ to $G$. Moreover, the Levi-Civita connection $\bar{\nabla}$ of $h$ induces $(\nabla, \nabla^*)$ via natural horizontal and vertical projection maps. Similar statements hold for the curvature.
\end{theorem}

This result embeds the dualistic geometry of information geometry, in complete generality, in the pseudo-Riemannian framework. Intuitively, this is a consequence of the fact that the cross-difference \eqref{eqn:cross.difference} is equal to the symmetrization of the $c$-divergence on the graph $G$ of the transport map (see Proposition \ref{prop:cross.as.symmetrization}). As a by-product, we obtain a new information-geometric interpretation of the MTW tensor on $G$ (Corollary \ref{cor:MTW.interpret}). Let us note that Theorem \ref{thm:sec3} is not the first result that considers the geometry of optimal transport in the context of information geometry. In \cite{KZ18} Khan and Zhang computed the MTW tensor for the logarithmic cost function corresponding to the $L^{(\alpha)}$-divergence, and found that it has a particularly simple form. This result motivated our work on this paper. In the follow up work \cite{KZ19}, they considered convex costs on $\mathbb{R}^n$ of the form $c(x, y) = \Psi(x - y)$, and showed that the MTW tensor is a multiple of the orthogonal holomorphic bisectional curvature of a K\"{a}hler manifold equipped with the Sasaki metric. On the other hand, in this paper we consider an arbitrary divergence and connect its dualistic structure with the original pseudo-Riemannian framework of Kim and McCann. Thus in this sense our results are more general. 

As in classical differential geometry, spaces of constant sectional curvatures are of special interest. In our context there are two kinds of sectional curvatures: the cross curvature of a cost function on $M \times M'$ (optimal transport) and the (primal and dual) sectional curvatures of the dualistic structure on $G$ (information geometry). Their relations are studied in Section 4. In particular, constant cross curvature implies constant information-geometric curvature. From this point of view, the dually flat geometry of Bregman divergence, which corresponds to the quadratic transport, follows immediately from the flatness of the pseudo-Euclidean space. The $L^{(\alpha)}$-divergence,  and its associated logarithmic cost function, lead to constant negative curvature.\footnote{Constant positive curvature corresponds to the $L^{(-\alpha)}$-divergence introduced in \cite{W18}. Since the proofs are similar, in this paper we focus on the $L^{(\alpha)}$-divergence.} In this setting we provide an intrinsic interpretation of the information geometric curvature in terms of the divergence between a pair of primal dual geodesics. Finally, in Section \ref{sec:conclusion} we conclude and discuss some avenues for future research.

We end the introduction with a brief discussion of the relevant literature. The connections between optimal transport and information geometry have been the inspiration of many recent papers. Apart from our line of works which centers on the $c$-divergence and $L^{(\alpha)}$-divergence corresponding to the Dirichlet transport, other perspectives have been considered in the literature and the following discussion is far from exhaustive. For example, \cite{AKO18, AKOC19} study divergences defined using the entropically relaxed transport problem (see Example \ref{eg:entropic}). The porous medium equation, which played an important role in the development of optimal transport, is studied in \cite{OW09} using tools of information geometry.  Relations between the Wasserstein metric and the Fisher-Rao metric are studied in \cite{M17, LM18, AKO18, CPSV18} among many others. Finite dimensional submanifolds of the Wasserstein space as well as their geometric and statistical properties are studied in a series of papers by Li and his collaborators; see \cite{li2018transport, CL18, LZ19} and the references therein. 

\section{$c$-divergence and the pseudo-Riemannian framework} \label{sec:prelim}
This section sets the stage of the paper. In Section \ref{sec:c.divergence} we review the $c$-divergence introduced by Pal and the first author in \cite{PW18, W18}. It is worth noting that {\it any} divergence in the sense of Definition \ref{def:divergence} can be regarded as a $c$-divergence via a suitable choice of the cost function (see Example \ref{eg:any.divergence}). Then, in Section \ref{sec:pseudo.Riemannian}, we introduce the pseudo-Riemannian framework of Kim and McCann \cite{KM10, M12}. For motivations and background in optimal transport we refer the reader to \cite{PW18, W18} and their references.

\subsection{$c$-divergence} \label{sec:c.divergence}
We first adapt the definitions in \cite{PW18, W18} to our differential-geometric setting. Let $M$ and $M'$ be $n$-dimensional smooth manifolds, and let $c: M \times M' \rightarrow \mathbb{R}$ be a (smooth) cost function. In many cases of interest we have $M = M' \subset \mathbb{R}^n$. We denote generic points in $M$ and $M'$ by $p$ and $q'$ respectively. 

Let $\varphi: M \rightarrow \mathbb{R}$ be smooth and $c$-concave, i.e., there exists $\psi : M' \rightarrow \mathbb{R} \cup \{-\infty\}$ such that
\[
\varphi(p) = \inf_{q' \in M'} (c(p, q') - \psi(q')), \quad p \in M.
\]
Suppose that the $c$-gradient of $\varphi$ exists, i.e., for each $p \in M$ there exists a unique $q' =: D^c \varphi (p) \in M'$ such that $\varphi(p) + \psi(q') = c(p, q')$. We assume that $f = D^c \varphi$ is a smooth diffeomorphism from $M$ onto its range $f(M) \subset M'$. Intuitively, $f$ represents an optimal transport map, with respect to the cost $c$, for a given pair of probability measures on $M$ and $M'$ respectively. Note that $\psi = \varphi^c$ is the $c$-transform of $\varphi$ and is given by
\begin{equation} \label{eqn:c.transform}
	\psi(q) = \inf_{p \in M} (c(p, q) - \varphi(p)) = c(f^{-1}(q), q) - \varphi(f^{-1}(q)), \quad q \in M'.
\end{equation}
For an exposition of these concepts see \cite[Section 1.6]{S15}.

Let $G$ be the graph of the optimal transport map, i.e.,
\begin{equation} \label{eqn:graph.transport.map}
	G = G_{f} := \{(p, f(p)) \in M \times M' : p \in M\}.
\end{equation}
Clearly $G$ is an $n$-dimensional embedded submanifold of the product manifold $M \times M'$. We regard the $c$-divergence, defined below, as a divergence on $G$. 

\begin{definition}[$c$-divergence] 
	For $x = (p, q), x' = (p', q') \in G$, we define
	\begin{equation} \label{eqn:c.divergence}
		{\bf D}[x : x'] = c(p, q') -  \varphi(p) - \psi(q'). 
	\end{equation}
	We call ${\bf D} : G \times G \rightarrow [0, \infty)$ the $c$-divergence by the transport on $G$.
\end{definition}

By the generalized Fenchel inequality ${\bf D}$ is non-negative. For it to be a divergence on $G$ in the sense of Definition \ref{def:divergence}, we require that the cost function is non-degenerate in the sense of \cite[Definition 2.2]{KM10}, which means that in local coordinates the 
matrix $\frac{\partial^2}{\partial p^i \partial q'^{j}} c(p, q')$ is invertible (see also the discussion in Section \ref{sec:pseudo.Riemannian}). This condition will be assumed throughout the paper. 

\begin{remark} \label{rem:graph.transport}
	Note that the $c$-divergence is defined on the graph $G$. We may identify $G$ with $M$ (or $f(M) \subset M'$) via the mapping $x = (p, f(p)) \mapsto p$ (or $x = (p, f(p)) \mapsto q = f(p)$). In particular, $M$ and $f(M)$ may also be identified via the transport map $p \mapsto f(p)$. The last identification is used in \cite[Definition 3.3]{PW18} and \cite[Definition 7]{W18}.
\end{remark}

\begin{remark}[Interpretation of the $c$-divergence] \label{rem:interpretation.c.div}
	Let $\mu$ and $\nu$ be probability measures satisfying $\nu = f_{\#}\mu$, so that $f = D^c \varphi$ is an optimal transport map for the pair $(\mu, \nu)$. By Kantorovich's duality, the value of the transport problem is given by
	\[
	\mathcal{T}_c(\mu, \nu) = \int_M \varphi(p) d\mu(p) + \int_{M'} \psi(q') d \nu(q').
	\]
	Now let $\gamma \in \Pi(\mu, \nu)$ be a coupling of $(\mu, \nu)$ which may be suboptimal. Given $(p, q')$ on the support of $\gamma$, let $x := (p, f(p)), x' := (f^{-1}(q'), q') \in G$; see the projection maps introduced in Definition \ref{def:projection.maps}. We have
	\begin{equation*}
		\begin{split}
			\int_{M \times M'} {\bf D}[x : x'] d\gamma &= \int c(p, q') - \varphi(p) - \psi(q') d\gamma \\
			&= \int c d \gamma - \mathcal{T}_c(\mu, \nu) \geq 0.
		\end{split}
	\end{equation*}
	Thus the expected $c$-divergence is the excess transport cost compared to the optimal coupling. Intuitively, the $c$-divergence ${\bf D}[x : x']$ measures the ``distance'' between $(p, q')$ and the graph $G$ of optimal transport (see Figure \ref{fig:projections}).
\end{remark}

Before giving specific examples, let us make the important observation that {\it any divergence can be regarded as a $c$-divergence}.

\begin{example}[Any divergence is a $c$-divergence] \label{eg:any.divergence}
Suppose $M = M'$ and ${\bf D}: M \times M \rightarrow [0, \infty)$ is an arbitrary divergence as in Definition \ref{def:divergence}. Consider the cost function $c \equiv {\bf D}$ given by the divergence. Then the identity transport $f(p) \equiv p$ has a graph -- the diagonal of $M \times M$ -- which is $c$-cyclically monotone. This transport map is induced by the constant $c$-concave function $\varphi(x) \equiv 0$ whose $c$-transform $\psi$ is also zero. Since $\varphi$ and $\psi$ both vanish, the $c$-divergence \eqref{eqn:c.divergence} is exactly the given divergence ${\bf D}$.
\end{example}

While any divergence can be regarded as a $c$-divergence, the more interesting case for the purposes of this paper is where a given cost induces many $c$-cyclically monotone graphs \eqref{eqn:graph.transport.map} by varying the transport map $f$ (which solves the optimal transport problem for different pairs of $(\mu, \nu)$). Clearly the existence of such graphs is closely related to the existence and regularity of optimal transport map. Instead of giving precise sufficient conditions (which can be found in \cite[Chapter 12]{V08}), to focus on the geometric ideas we simply assume that $f$ and the Kantorovich potentials $(\varphi, \psi)$ are given.

\begin{example} [Bregman divergence] \label{eg:Bregman.divergence}
Let $M = M' = \mathbb{R}^n$ and let $c(p, q') = \frac{1}{2}|p - q'|^2$ be the quadratic cost (where $p$, $q'$ are expressed in Euclidean coordinates). Note that for this cost the effective term for the transport is $-p \cdot q'$ since the other two terms of the expansion depend only on one of the variables. This important example has been treated in \cite[Section 2]{W18} and to illustrate the ideas we recall the argument. By Brenier's theorem \cite{B91}, the optimal transport map has the form $q = f(p) = D \phi(p)$, where $\phi$ is a convex function and $D\phi$ is the gradient. Here we assume $\phi$ is smooth and $\phi'' > 0$. Then $\varphi(p) = \frac{1}{2}|p|^2 - \phi(p)$ is $c$-concave, $D^c \varphi = D \phi$, and $\psi(q') = \varphi^c(q') = \frac{1}{2}|q'|^2 - \varphi^*(q')$, where $\varphi^*$ is the convex conjugate of $\varphi$. Since $\varphi(p) + \varphi^*(q') \equiv p \cdot q'$ by convex duality, substituting into \eqref{eqn:c.divergence} and simplifying, we see that the $c$-divergence corresponding to the triple $(f, \varphi, \psi)$ is the classic Bregman divergence:
\begin{equation} \label{eqn:c.as.Bregman1}
{\bf D}[x : x'] = (\phi(p) - \phi(p')) - D \phi(p') \cdot ( p - p'),
\end{equation}
where $p' = f^{-1}(q') = D \phi^*(q')$. Note that in \eqref{eqn:c.as.Bregman1} we represented the divergence in terms of the primal coordinate system $p$. When expressed in the dual coordinate system $q$, we have
\begin{equation*} 
{\bf D}[x : x'] = (\psi(q') - \psi(q)) - D \psi(q) \cdot ( q' - q).
\end{equation*}
Thus we recover the self-dual representation of the Bregman divergence \cite[(1.68)]{A16}. Its information geometry, which is dually flat \cite[Chapter 1]{A16}, is revisited in Example \ref{eg:Bregman.geometry}. 
\end{example}

\begin{example} [$L^{(\alpha)}$-divergence] \label{ex:L.alpha}
This may be regarded as a ``nonlinear deformation'' of Example \ref{eg:Bregman.divergence}. Let $M = M' = (0, \infty)^n$, and let $\alpha > 0$ be a fixed parameter. Consider the logarithmic cost function
\begin{equation} \label{eqn:alpha.cost}
c(p, q') = \frac{1}{\alpha} \log (1 + \alpha p \cdot q'),
\end{equation}
where $a \cdot b$ is the Euclidean dot product. This reduces, after a suitable reparameterization, to the cost function of the Dirichlet transport studied in \cite{PW18b}. Remarkably, as shown in \cite{PW18b} and \cite[Theorem 6]{W18}, the optimal transport map is still explicit and is given by
\begin{equation} \label{eqn:transport.map.L.alpha}
f(p) = \frac{D\varphi(p)}{1 - \alpha D\varphi(p) \cdot p},
\end{equation}
where $\varphi$ is an $\alpha$-exponentially concave function on $M$ (i.e., $e^{\alpha \varphi}$ is concave) satisfying suitable regularity conditions (see \cite[Condition 7]{W18}) that will be assumed implicitly.  The representation \eqref{eqn:transport.map.L.alpha} of the optimal transport map may be regarded as an analogue of Brenier's theorem. By an argument similar to the one presented in Example \ref{eg:Bregman.divergence}, it can be shown that the $c$-divergence is the $L^{(\alpha)}$-divergence given by
\begin{equation} \label{eqn:L.alpha.divergence}
{\bf D}[x : x'] = \frac{1}{\alpha} \log ( 1 + \alpha D \varphi(p') \cdot (p - p')) - (\varphi(p) - \varphi(p')).
\end{equation}
For detailed studies of this divergence see \cite{PW18, W15, W18, PW18b, W19}. Note that as $\alpha \rightarrow 0$ the $L^{(\alpha)}$-divergence converges to the Bregman divergence of the convex function $-\varphi$. In \cite{PW18, W18} it was shown that the induced dualistic structure is dually projectively flat with constant sectional curvature $-\alpha$ (the converse is also true; for the precise statement see \cite[Theorem 19]{W18}). Also see \cite{P17} which interprets the quadratic cost and \eqref{eqn:alpha.cost} in terms of convex costs (Example \ref{eg:convex.cost}) defined by exponential families. The regularity theory of this transport problem is recently addressed in \cite{KZ18} which motivated our study. 
\end{example}

\begin{example} [Entropic regularization] \label{eg:entropic}
	From \eqref{eqn:c.divergence}, a general $c$-divergence is, apart from a change of coordinates (via $q = f(p)$), the same as the original cost function up to some linear terms that involve only $p$ or $q'$. We show how this idea can be used to interpret the $D_{\lambda}$-divergence in the recent paper \cite{AKOC19} which defines a ``modified Sinkhorn divergence'' for the entropically regularized transport problem. 
	
	To be consistent with this paper we modify slightly the notations of \cite{AKOC19}. Let $\mathcal{X} = \{0, 1, \ldots, n\}$ be a finite set, and let $C = (C_{ij})$ be a non-negative cost on $\mathcal{X} \times \mathcal{X}$ which vanishes on the diagonal. Given $p, q' \in \mathcal{P}(\mathcal{X})$ (probability distributions on $\mathcal{X}$) and a coupling $\pi \in \Pi(p, q')$, consider the entropically regularized cost given by 
	\[
	\mathcal{L}(\pi) = C_{i, j} p^i q'^j - \lambda H(\pi),
	\]
	where $\lambda > 0$ is a regularization parameter and $H(\pi)$ is the Shannon entropy of the coupling $\pi$. We define the so-called $C$-function by
	\begin{equation} \label{eqn:C.function}
		C_{\lambda} (p, q') := \min_{\pi \in \Pi(p, q')} \mathcal{L}(\pi).
	\end{equation}
	Let $M = M' = \mathcal{P}_+ (\mathcal{X}) \equiv \{p \in (0, 1)^{1 + n} : \sum_i p^i = 1\}$ and consider the cost function $c = C_{\lambda}$. In \cite[Theorem 2]{AKOC19} it is shown that
	\begin{equation} \label{eqn:shrinkage.operator}
		q'^* := \argmin_{q' \in M'} C_{\lambda}(p, q') = \tilde{K}_{\lambda} p,
	\end{equation}
	where $\tilde{K}_{\lambda}$ is an injective shrinkage operator. 
	
	Consider the function $\psi(q') \equiv 0$ on $M'$. Then
	\[
	\varphi(p) = \psi^c(p) = \inf_{q' \in M'} C_{\lambda}(p, q') = C_{\lambda}(p, \tilde{K}_{\lambda}p)
	\]
	and it is easy to see that $\varphi^c = \psi^{cc} = 0 = \psi$ on the range of $\tilde{K}_{\lambda}$. Thus, restricting $c$ to $M \times \tilde{K}_{\lambda}(M)$, $\varphi$ is $c$-concave and the corresponding optimal transport map is given by $q = f(p) = \tilde{K}_{\lambda} p$. The $c$-divergence is given by
	\[
	{\bf D}[x : x'] = C_{\lambda}(p, q') - C_{\lambda}(p, \tilde{K}_{\lambda}p) = C_{\lambda}(p, \tilde{K}_{\lambda}p') - C_{\lambda}(p, \tilde{K}_{\lambda}p),
	\]
	which is nothing but the $D_{\lambda}$-divergence in \cite[Definition 1]{AKOC19} apart from a multiplicative constant. Clearly the same approach extends to other regularizations as long as the analogue of \eqref{eqn:shrinkage.operator} is well-defined. 
\end{example}

\subsection{Pseudo-Riemannian framework} \label{sec:pseudo.Riemannian}
Next we describe the pseudo-Riemannian framework of Kim and McCann \cite{KM10} which gives a geometric interpretation of the regularity theory of optimal transport studied by Ma, Trudinger \& Wang \cite{MTW05}, Loeper \cite{L09} and many others. A general reference of pseudo-Riemannian geometry is \cite{O83}.

Consider manifolds $M, M'$ and the cost function $c$ as above. Recall the cross-difference $\delta = \delta(p, q', p_0, q_0') : (M \times M')^2 \rightarrow \mathbb{R}$ defined in \eqref{eqn:cross.difference}. 
The pseudo-Riemannian metric $h$ is given in terms of the Hessian of $\delta$ in the $2n$-dimensional variables $(p, q')$. More precisely, let $\xi = (\xi^1, \ldots, \xi^n)$ and $\eta' = (\eta'^{\bar{1}}, \ldots, \eta'^{\bar{n}})$ be local coordinates on $M$ and $M'$ respectively, and express the cost function in the form $c = c(\xi, \eta')$. Note that we use $i$ to denote indices for $M$ and $\bar{i}$ for $M'$ to be consistent with the index notations in \cite{KM10}. We denote
\begin{equation} \label{eqn:c.derivatives}
	\begin{split}
		c_{i:}(\xi, \eta') = \frac{\partial}{\partial \xi^i} c(\xi, \eta'), &\quad c_{:\bar{j}}(\xi, \eta') = \frac{\partial}{\partial \eta'^{\bar{j}}} c(\xi, \eta'),\\
		c_{i:\bar{j}}(\xi, \eta') = \frac{\partial^2 }{\partial \xi^i\partial \eta'^{\bar{j}}} c(\xi, \eta'), &\quad 
		c_{ij:\bar{k}}(\xi, \eta') = \frac{\partial^2 }{\partial \xi^i \partial \xi^j} \frac{\partial}{\partial \eta'^{\bar{k}}} c(\xi, \eta'),
	\end{split}
\end{equation}
and so on. By the product structure we have the canonical decomposition
\[
T_{(p, q')} (M \times M') = T_{p} M \oplus T_{q'} M'.
\]
A generic tangent vector $v$ at $(\xi, \eta') \in M \times M'$ (with an abuse of notation) can be written in the form
\begin{equation} \label{eqn:tangent.vector.product}
	v = a^i \frac{\partial}{\partial \xi^i} + b^{\bar{i}} \frac{\partial}{\partial \eta'^{\bar{i}}}.
\end{equation}
Define the $2n \times 2n$ matrix
\begin{equation} \label{eqn:metric.h}
	h = h(\xi, \eta') = 
	\frac{1}{2} 
	\begin{bmatrix}
		0       & -\bar{D}D c \\
		-D\bar{D} c       & 0 
	\end{bmatrix},
\end{equation}
where the matrices $\bar{D}D c := (c_{i:\bar{j}})_{i,\bar{j}}$ and $\bar{D}Dc := (c_{j:\bar{i}})_{\bar{i},j}$ are evaluated at $(\xi, \eta')$. We assume throughout that $c$ is non-degenerate, i.e., $\overline{D}Dc$ and $D \overline{D}c$ are invertible. Using \eqref{eqn:tangent.vector.product} and \eqref{eqn:metric.h}, we define a pseudo-Riemannian metric $h$ by
\begin{equation} \label{eqn:pseudo.Riemannian.metric}
	h(\xi, \eta')(v, v) = v^{\top} h v = \frac{1}{2}
	\begin{bmatrix}
		a^{\top}       & b^{\top}  
	\end{bmatrix}
	\begin{bmatrix}
		0       & -\bar{D}D c \\
		-D\bar{D} c       & 0 
	\end{bmatrix}
	\begin{bmatrix}
		a \\
		b  
	\end{bmatrix} = -c_{i:\bar{j}} a^i b^{\bar{j}}.
\end{equation}
It is easy to see that \eqref{eqn:pseudo.Riemannian.metric} is equivalent to the following intrinsic definition.

\begin{lemma} \label{lem:metric.intrinsic}
	Write the cross-difference in the form $\delta(p, q', p_0, q_0') = \delta(x, y)$, where $x = (p, q'), y = (p_0, q_0') \in M \times M'$. Let $X$ and $Y$ be vector fields on $M \times M'$. Then
	\[
	h(X, Y) = -\left.X_{(x)} Y_{(y)} \delta(x, y)\right|_{x = y},
	\]
	where $X_{(x)}$ is the derivation applied to the function when $x$ varies and $y$ is kept fixed (similar for $Y_{(y)}$).
\end{lemma}

In \cite{KM10} it is shown that $h$ has signature $(n, n)$, i.e., the matrix  \eqref{eqn:metric.h} (denoted also by $h$) has $n$ positive eigenvalues and $n$ negative eigenvalues. Given the pseudo-Riemannian metric $h$, one can consider geodesics with respect to the Levi-Civita connection $\bar{\nabla}$ as well as the Riemann curvature tensor $\bar{R}$; these objects will be studied in the next section. In particular, the (unnormalized) sectional curvature gives a geometric interpretation of the Ma-Trudinger-Wang (MTW) tensor, introduced in \cite{MTW05}, which plays a crucial role in the regularity theory of optimal transport maps. We will recall the definition of the MTW tensor in Remark \ref{rmk:MTW}.

The following result explains intuitively why the dualistic structure and the pseudo-Riemannian framework are related. The details of this relation, which amounts to desymmetrizing \eqref{eqn:c.divergence.and.cross.diff}, are worked out in Section \ref{sec:geometry.connection}.

\begin{proposition}[Cross difference is symmetrization of $c$-divergence] \label{prop:cross.as.symmetrization}
Consider a $c$-divergence ${\bf D}$ associated to $(f, \varphi, \psi)$ and the graph $G$. Then for $x = (p, q)$ and $x' = (p', q')$ in $G$, we have
\begin{equation} \label{eqn:c.divergence.and.cross.diff}
\delta(p, q, p', q') = {\bf D}[x : x'] + {\bf D}[x' : x].
\end{equation}
Thus on $G \times G$ the cross difference is equal to the symmetrization of the $c$-divergence.
\end{proposition}
\begin{proof}
Let $x = (p, q), x' = (p', q') \in G$. Using the definition of $c$-divergence, we have
\begin{equation*} 
\begin{split}
{\bf D}[x : x'] + {\bf D}[x' : x] &= (c(p, q') - \varphi(p) - \psi(q')) + (c(p', q) - \varphi(p') - \psi(q)) \\
	&= c(p, q') + c(p', q) - c(p, q) - c(p', q') = \delta(p, q, p', q').
\end{split}
\end{equation*}
\end{proof}

\begin{remark}
As an extension of \eqref{eqn:c.divergence.and.cross.diff}, we may consider three pairs of points instead of two. Given $x_i = (p_i, q_i)$, $i = 1, 2, 3$, on $G$, we have
\begin{equation} \label{eqn:3.points}
{\bf D}[x_2 : x_1] + {\bf D}[x_3 : x_2] - {\bf D}[x_3 : x_1] = c(p_2, q_1) + c(p_3, q_2) - c(p_3, q_1) - c(p_2, q_2).
\end{equation}
This identity was first observed in \cite[Section 3.3]{PW18}. In terms of optimal transport, this equals the excess transport cost (which can be positive or negative) of the coupling $(p_1 \rightarrow q_3, p_2 \rightarrow q_1, p_3 \rightarrow q_2)$ over $(p_1 \rightarrow q_3, p_2 \rightarrow q_2, p_3 \rightarrow q_1)$, and reduces to the cross-difference when $x_1 = x_3$. Since a divergence is locally quadratic (see \eqref{eqn:divergence.local.quadratic}), the left hand side of \eqref{eqn:3.points} may be called a ``Pythagorean expression". Such expressions play an important role in information geometry. Specifically, both the Bregman and $L^{(\alpha)}$-divergence satisfy a {\it generalized Pythagorean theorem} \cite[Theorem 16]{W18} which characterizes the sign of \eqref{eqn:3.points} in terms of the Riemannian angle of a primal-dual geodesic triangle.
\end{remark}


\begin{example} [Quadratic cost] \label{eg:quadratic.cost}
	Suppose $M = M' = \mathbb{R}^n$ and $c(p, q') = \frac{1}{2}|p - q'|^2$ as in Example \ref{eg:Bregman.divergence}. It is easy to verify that the matrix of the pseudo-Riemannian metric is given by
	\begin{equation} \label{eqn:metric.h.quadratic.case}
		h = h(p, q') = 
		\frac{1}{2} 
		\begin{bmatrix}
			0       & I \\
			I       & 0 
		\end{bmatrix},
	\end{equation}
	where $I$ is the $n \times n$ identity matrix. If $v = a^i \frac{\partial}{\partial p^i} + b^{\bar{i}} \frac{\partial}{\partial q'^{\bar{i}}}$ is a tangent vector, then \eqref{eqn:metric.h.quadratic.case} gives $h(p, q')(v, v) = a^1 b^{\bar{1}} + \cdots + a^n b^{\bar{n}}$.
	
	Consider the pseudo-Euclidean space $\mathbb{R}_n^{2n} := \{(x, y): x, y \in \mathbb{R}^n\}$ with the metric
	\begin{equation} \label{eqn:pseudo.Euclidean.space.metric}
		ds^2 = (x^1)^2 + \cdots + (x^n)^2 - (y^{1})^2 - \cdots - (y^{n})^2.
	\end{equation}
	It is easy to verify that the mapping
	\[
	(p, q') \mapsto \left(\frac{p^1 + q'^1}{2}, \ldots, \frac{p^n + q'^n}{2}, \frac{p^1 - q'^1}{2}, \ldots, \frac{p^n - q'^n}{2}\right)
	\]
	is an isometry from $(M \times M', h)$ to $\mathbb{R}_n^{2n}$. Note that $\mathbb{R}_n^{2n}$ is, up to isometries, the unique space form (complete connected pseudo-Riemannian manifold with constant curvature) with signature $(n, n)$ and zero curvature; see \cite[Corollary 8.24]{O83}. As we shall see in Example \ref{eg:Bregman.geometry}, the dually flat geometry of Bregman divergence follows directly from our framework and the flatness of the pseudo-Euclidean space.  
\end{example}

\begin{example} [General convex cost] \label{eg:convex.cost}
	Suppose again $M = M' = \mathbb{R}^n$. Now let $c(p, q') = \Psi(p - q')$ where $\Psi$ is strictly convex. The solution to this transport problem is given by Gangbo and McCann \cite{GM96}. The pseudo-Riemannian metric is given in Euclidean coordinates by
	\[
	h = h(p, q') = 
	\frac{1}{2} 
	\begin{bmatrix}
		0       & D^2 \Psi (p - q') \\
		D^2 \Psi  (p - q')     & 0 
	\end{bmatrix},
	\]
	where $D^2 \Psi$ is the Hessian matrix of $\Psi$. Khan and Zhang \cite{KZ19} expressed the MTW tensor for this cost in terms of the bisectional curvature of a certain K\"{a}lher manifold. 
\end{example}

\section{Connecting the two geometries} \label{sec:geometry.connection}
We show that the pseudo-Riemannian framework encodes the dualistic structure in information geometry. In essence, the pseudo-Riemannian metric $h$ on the product manifold $M \times M'$ induces, in a sense to be made precise, the dualistic structure $(g, \nabla, \nabla^*)$ of the $c$-divergence on the graph $G$ regarded as a submanifold of $M \times M'$.

\subsection{Preliminaries} \label{sec:coeff.prelim}
We begin with some notations and preliminary results. Consider the graph
\[
G = \{(p, f(p)) : p \in M\} \subset M \times M'
\]
equipped with the $c$-divergence ${\bf D}$ given by \eqref{eqn:c.divergence}. Fix local coordinate systems $\xi = (\xi^1, \ldots, \xi^n)$ on $M$ and $\eta' = (\eta'^{\bar{1}}, \ldots, \eta'^{\bar{n}})$ on $M'$. Clearly $(\xi, \eta')$ is a coordinate system on $M \times M'$. By an abuse of notations, the coordinates are related on $G$ by $\eta = f(\xi)$. Projecting $G$ to $M$ (using $(p, f(p)) \mapsto p$) and $M'$ (using $(p, f(p)) \mapsto q = f(p)$) respectively, we may regard $\xi$ and $\eta$ as local coordinate systems of $G$. We call $\xi$ the {\it primal} coordinates and $\eta$ the {\it dual} coordinates on $G$. See Figure \ref{fig:primal.dual.coordinates} for an illustration (also see Remark \ref{rem:graph.transport} and compare with \cite[Figure 2]{PW18}).

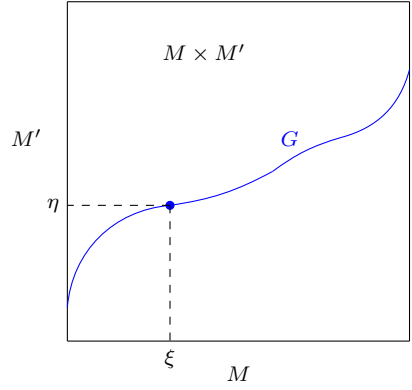
\begin{figure}[t!]
	\centering
	\begin{tikzpicture}[scale = 0.45]
		\draw (0, 0) to (10, 0);
		\draw (10, 0) to (10, 10);
		\draw (10, 10) to (0, 10);
		\draw (0, 10) to (0, 0);
		
		
		\draw[blue] (0, 1) to [bend left=40] (3, 4);
		\draw[blue] (3, 4) to [bend left=-10] (6, 5);
		\draw[blue] (6, 5) to [bend left=10] (8, 6);
		\draw[blue] (8, 6) to [bend left = -30] (10, 8);
		
		\node [blue, above] at (6.5, 5.5) {$G$};
		\node[circle, draw = blue, fill = blue,  inner sep=0pt, minimum size=3pt] at (3, 4) {};
		\draw[black, dashed] (3, 4) to (3, 0);
		\node [black, left] at (0, 4) {$\eta$};
		\draw[black, dashed] (3, 4) to (0, 4);
		\node [black, below] at (3, 0) {$\xi$};	
		
		\node [black, below] at (4, 9) {$M \times M'$};	
		\node [black, below] at (5, -0.5) {$M$};	
		\node [black, left] at (-0.5, 6) {$M'$};	
		
	\end{tikzpicture}
	\caption{The graph $G$ of optimal transport as an $n$-dimensional submanifold of $M \times M'$. The primal and dual coordinates correspond to projections onto $M$ and $f(M) \subset M'$ respectively.} \label{fig:primal.dual.coordinates}
\end{figure}

As mentioned in Section \ref{sec:summary}, the $c$-divergence ${\bf D}$, given by \eqref{eqn:c.divergence}, induces a dualistic structure $(g, \nabla, \nabla^*)$ on $G$, where $g$ is a Riemannian metric and $\nabla$ and $\nabla^*$ are torsion-free affine connections. Let us give the coordinate representation of these objects (see \cite[Chapter 6]{A16} and \cite[Chapter 11]{CU14} for more details). Suppose we use the primal coordinate system $\xi$ on $G$. Writing ${\bf D} = {\bf D}[\xi : \xi']$ as a function of $(\xi, \xi')$, we have
\begin{equation} \label{eqn:dualistic.structure.coefficients}
	\begin{split}
		g_{ij}(\xi) &= - \left. \frac{\partial}{\partial \xi^i} \frac{\partial}{\partial \xi'^j} \mathbf{D} [\xi : \xi'] \right|_{\xi = \xi'},\\
		\Gamma_{ijk}(\xi) &= - \left. \frac{\partial^2}{\partial \xi^i \partial \xi^j} \frac{\partial}{\partial \xi'^k} \mathbf{D} [\xi : \xi'] \right|_{\xi = \xi'},\\
		\Gamma_{ijk}^*(\xi) &= - \left. \frac{\partial^2}{\partial \xi'^i \partial \xi'^j} \frac{\partial}{\partial \xi^k} \mathbf{D} [\xi : \xi'] \right|_{\xi = \xi'}.\\
	\end{split}
\end{equation}
Here $\Gamma_{ijk}$ and $\Gamma_{ijk}^*$ are the Christoffel symbols of $\nabla$ and $\nabla^*$ respectively. Also we define $\Gamma_{ij}\mathstrut^{k} (\xi) = \Gamma_{ijm}(\xi) g^{mk}(\xi)$ and $\Gamma^*_{jk}\mathstrut^{k} (\xi) = \Gamma^*_{ijm}(\xi) g^{mk}(\xi)$, where $(g^{ij})$ is the inverse matrix of $(g_{ij})$. For instance, if $\partial_i = \partial/\partial \xi^i$ we have $\nabla_{\partial_i} \partial_j = {\Gamma_{ij}}^k \partial_k$. Similarly, we can write down the coefficients in terms of the dual coordinates $\eta = f(\xi)$. We denote by $\frac{\partial \eta}{\partial \xi} = \left( \frac{\partial \eta^{\bar{i}}}{\partial \xi^j} \right)$ the Jacobian matrix of the transition map $\xi \mapsto \eta$. Its inverse is given by $\frac{\partial \xi}{\partial \eta} = \left( \frac{\partial \xi^i}{\partial \eta^{\bar{j}}} \right)$. 

Write $c = c(\xi, \eta')$ as a function of $(\xi, \eta')$ (locally in $M \times M'$) and consider the matrix $(c_{i:\bar{j}})$ of cross derivatives given by \eqref{eqn:c.derivatives}. We denote its inverse (which exists since $c$ is non-degenerate) by $(c^{\bar{i} : j})$.  This means that
\begin{equation} \label{eqn:c.matrix.inverse}
	c_{i : \bar{m}}(\xi, \eta') c^{\bar{m} : j}(\xi, \eta') = \delta_i^j, \quad c^{\bar{i}:m}(\xi, \eta') c_{m:\bar{j}}(\xi, \eta') = \delta_{\bar{j}}^{\bar{i}},
\end{equation}
where $\delta_i^j$ and its analogues are Kronecker deltas. Differentiating \eqref{eqn:c.matrix.inverse}, we obtain the following useful identities that are also used in \cite{KM10}.

\begin{lemma} \label{lem:c.inverse.derivative}
	Under the local coordinate system $(\xi, \eta')$ in $M \times M'$, We have
	\begin{equation} \label{eqn:c.inverse.derivatives}
		\frac{\partial}{\partial \xi^k} c^{\bar{\ell}:j}(\xi, \eta') = -c^{\bar{\ell}:i} c_{ik:\bar{m}} c^{\bar{m}:j}, \quad \frac{\partial}{\partial \eta'^{\bar{k}}} c^{\bar{\ell}:j}(\xi, \eta') = -c^{\bar{\ell}:i} c_{i:\bar{k}\bar{m}} c^{\bar{m}:j}.
	\end{equation}
\end{lemma}

With these notations we are ready to express the dualistic structure of the $c$-divergence ${\bf D}$ on $G$. 

\begin{lemma} \label{lem:info.geo.coefficients}
	Under the primal coordinate system $\xi$ of $G$, we have
	\begin{equation} \label{eqn:c.divergence.coefficients}
		\begin{split}
			g_{ij}(\xi) = -c_{i : \bar{m}} \frac{\partial \eta^{\bar{m}}}{\partial \xi^j}, &\quad 	g^{ij}(\xi) = -\frac{\partial \xi^i}{\partial \eta^{\bar{m}}} c^{\bar{m}: j}, \\
			\Gamma_{ijk}(\xi) = - c_{ij:\bar{m}} \frac{\partial \eta^{\bar{m}}}{\partial \xi^k}, &\quad 	\Gamma_{ij}\mathstrut^{k}(\xi) = c_{ij:\bar{m}} c^{\bar{m}:k}.
		\end{split}
	\end{equation}
	where mixed derivatives such as $c_{i : \bar{m}}$ are evaluated at $(\xi, \eta) = (\xi, \eta(\xi))$, so that the coefficients are functions of $\xi$.
	
	Similarly, the coefficients of $g$ and $\nabla^*$ under the dual coordinate system $\eta$ are given as follows:
	\begin{equation}  \label{eqn:c.divergence.coefficients2}
		\begin{split}
			g_{\bar{i}\bar{j}}(\eta) =  - \frac{\partial \xi^m}{\partial \eta^{\bar{j}}} c_{m: \bar{i}},  &\quad g^{\bar{i}\bar{j}}(\eta) = - c^{\bar{j} : m} \frac{\partial \eta^{\bar{i}}}{\partial \xi^m}, \\
			\Gamma_{\bar{i}\bar{j}\bar{k}}^*(\eta) = -\frac{\partial \xi^m}{\partial \eta^{\bar{k}}} c_{m:\bar{i}\bar{j}}, &\quad \Gamma_{\bar{i}\bar{j}}^*\mathstrut^{\bar{k}}(\eta) = c^{\bar{k}:m} c_{m:\bar{i}\bar{j}} .
		\end{split}
	\end{equation}
\end{lemma}
\begin{proof}
	Express the $c$-divergence \eqref{eqn:c.divergence} in terms of local coordinate. Then \eqref{eqn:c.divergence.coefficients} and \eqref{eqn:c.divergence.coefficients2} follow from the definition \eqref{eqn:dualistic.structure.coefficients} via direct differentiation. Computations for the special case where ${\bf D}$ is an $L^{(\alpha)}$-divergence can be found in \cite{PW18, W18}.
\end{proof}

Note that $(g_{ij})$ and $(g^{\bar{i}\bar{j}})$ in \eqref{eqn:c.divergence.coefficients} and \eqref{eqn:c.divergence.coefficients2} are by construction symmetric even though this may not be apparent from the formulas. Next we consider the pseudo-Riemannian metric $h$ introduced in Section \ref{sec:pseudo.Riemannian}. The following result is taken from \cite[Lemma 4.1]{KM10}.

\begin{lemma} \label{lem:Levi.Civita.connection}
	Equip the product manifold $M \times M'$ with the pseudo-Riemannian metric $h$. Let $\bar{\nabla}$ be the Levi-Civita connection induced by $h$ and let $\bar{\Gamma}_{\cdot \cdot}\mathstrut^{\cdot}$ be its Christoffel symbols. In the local coordinates $\xi = (\xi^1, \ldots, \xi^n)$ for $M$ and $\eta' = (\eta'^{\bar{1}}, \ldots, \eta'^{\bar{n}})$ for $M'$, the only non-vanishing Christoffel symbols are
	\begin{equation} \label{eqn:h.Christoffel.symbols}
		\bar{\Gamma}_{ij}\mathstrut^{k}(\xi, \eta') = c_{ij:\bar{m}} c^{\bar{m}:k}  \quad \text{and}  \quad \bar{\Gamma}_{\bar{i}\bar{j}}\mathstrut^{\bar{k}}(\xi, \eta') =  c^{\bar{k}:m} c_{m:\bar{i}\bar{j}},
	\end{equation}
	where the derivatives are evaluated at $(\xi, \eta')$.
\end{lemma}


\subsection{Metrics and connections}
We are now ready to connect the two geometries, namely $(G, g, \nabla, \nabla^*)$ and $(M \times M', h)$. We first give two results concerning the metrics and the connections that are intuitive and easy to state; the curvature tensors will be studied in Section \ref{sec:curvature.tensors}.

First we consider the metrics. Recall that we have the canonical inclusion and decomposition
\begin{equation} \label{eqn:tangent.space.decomp}
	T_{(\xi, \eta)} G \subset T_{(\xi, \eta)} (M \times M') \equiv T_{\xi} M \oplus T_{\eta} M',
\end{equation}
where we again abuse notations and identify points with their coordinates. A generic element $v$ of $T_{(\xi, \eta)} G$ has the form
\begin{equation} \label{eqn:tangent.vector.submanifold}
	v = a^i  \left.\frac{\partial}{\partial \xi^i}\right|_{\xi} + a^i \frac{\partial \eta^{\bar{j}}}{\partial \xi^i} \left. \frac{\partial}{\partial \eta'^{\bar{j}}} \right|_{\eta},
\end{equation}
where $(a^1, \ldots, a^n) \in \mathbb{R}^n$ and $\frac{\partial \eta}{\partial \xi}$ is the Jacobian of the coordinate expression of $f$.

\begin{theorem} [$g$ as restriction of $h$ to $G$] \label{thm:metric}
	For any $(p, f(p)) \in G$ we have
	\begin{equation} \label{eqn:metric.restriction}
		\left. h \right|_{ (T_{(p, f(p)))} G)^2} = g.
	\end{equation}
	Thus the information geometric Riemannian metric $g$ is the restriction of $h$ to $G$.
\end{theorem}
\begin{proof}
	Simply evaluate \eqref{eqn:pseudo.Riemannian.metric} where $\eta' = \eta$ and $v$ is given by \eqref{eqn:tangent.vector.submanifold}, and compare with \eqref{eqn:c.divergence.coefficients}. (Also see \eqref{eqn:c.divergence.and.cross.diff}.)
\end{proof}

Next consider the primal and dual connections $\nabla$, $\nabla^*$ on $G$ as well as  the Levi-Civita connection $\bar{\nabla}$ on $M \times M'$.

\begin{definition} [Projection maps] \label{def:projection.maps}
	We define projection maps $\pi_0, \pi_1 : M \times f(M) \rightarrow G$ by
	\begin{equation} \label{eqn:projection.maps}
		\pi_0(p, q') = (p, f(p)), \quad  \pi_1(p, q') = (f^{-1}(q'), q'),
	\end{equation}
	for $x = (p, q') \in M \times f(M)$.
\end{definition}

\begin{figure}[t!]
	\centering
	\begin{tikzpicture}[scale = 0.45]
		\draw (0, 0) to (10, 0);
		\draw (10, 0) to (10, 10);
		\draw (10, 10) to (0, 10);
		\draw (0, 10) to (0, 0);
		
		
		\draw[blue] (0, 1) to [bend left=40] (3, 4);
		\draw[blue] (3, 4) to [bend left=-10] (6, 5);
		\draw[blue] (6, 5) to [bend left=10] (8, 6);
		\draw[blue] (8, 6) to [bend left = -30] (10, 8);
		
		\node [blue, above] at (6.5, 5.5) {$G$};
		\node[circle, draw = black, fill = black,  inner sep=0pt, minimum size=3pt] at (3.5, 5) {};
		\draw[black, dashed] (3.5, 5) to (3.5, 4.05);
		\node[circle, draw = black, fill = black,  inner sep=0pt, minimum size=3pt] at (3.5, 4.05) {};
		\draw[black, dashed] (3.5, 5) to (6, 5);
		\node[circle, draw = black, fill = black,  inner sep=0pt, minimum size=3pt] at (6, 5) {};
		\node[above] at (3.5, 5) {$x = (p, q')$};
		\node[below] at (3.5, 4.05) {$\pi_0(x)$};	
		\node[right] at (6, 4.8) {$\pi_1(x)$};	
		
		\node [black, below] at (5, 0) {$M$};	
		\node [black, left] at (0, 5) {$M'$};	
		
		
	\end{tikzpicture}
	\caption{The projection maps $\pi_0, \pi_1 : M \times f(M) \rightarrow G$.} \label{fig:projections}
\end{figure}
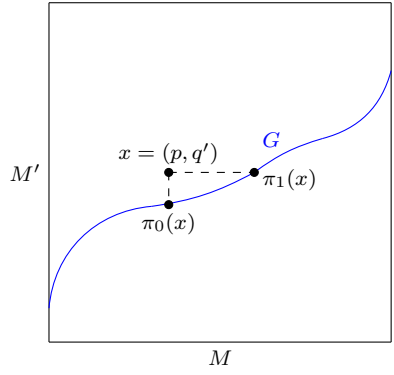

See Figure \ref{fig:projections} for an illustration. Motivated by this figure, we think of $\pi_0$ as the {\it vertical projection} and $\pi_1$ as the {\it horizontal projection} onto $G$. Note that if $v = v_0 \oplus v_1$ as in \eqref{eqn:tangent.space.decomp}, then $(\pi_i)_* v = v_i$, $i = 0, 1$ (the differential map).

Given a mapping $\pi : M \times f(M) \rightarrow G$ (which in our case is the projection $\pi_0$ or $\pi_1$) and the connection $\bar{\nabla}$ on $M \times M'$, we define an induced connection $\bar{\nabla}^{\pi}$ on $G$ as follows. Let $X, Y$ be given vector fields on $G$. For $x$ of $G$ fixed, we may extend $X$ and $Y$ to vector fields $\tilde{X}, \tilde{Y}$ in a neighborhood in $M \times M'$. So we may apply $\bar{\nabla}$ to $(\tilde{X}, \tilde{Y})$ near $x$ in $M \times \bar{M}$. Note that $ \bar{\nabla}_{\tilde{X}} \tilde{Y} |_x \in T_x (M \times M')$ is not necessarily tangent to $G$. We define
\begin{equation} \label{eqn:induced.connection}
	\left. \bar{\nabla}_X^{\pi} Y \right|_x  := \pi_*(  \bar{\nabla}_{\tilde{X}} \tilde{Y} |_x ) \in T_x G,
\end{equation}
where $\pi_* : T(M \times M') \rightarrow TG$ is the differential of $\pi$. Since $X$ and $Y$ are tangent to $G$, it can be verified easily that \eqref{eqn:induced.connection} defines unambiguously a torsion-free affine connection on $G$. 

\begin{theorem} \label{thm:induced.connections}
	We have $\nabla = \bar{\nabla}^{\pi_0}$ and $\nabla^* = \bar{\nabla}^{\pi_1}$. \end{theorem}
\begin{proof}
	Consider the coordinate system $(\xi, \eta')$ on $M \times M'$ as in Section \ref{sec:coeff.prelim}. Write $\partial_i = \frac{\partial}{\partial \xi^i}$ and $\partial_{\bar{i}} = \frac{\partial}{\partial \eta'^{\bar{i}}}$. Let $X$ and $Y$ be vector fields on $G$, and let their local extensions in $M \times M'$ be 
	\[
	\tilde{X} = \tilde{X}^i \partial_i +  \tilde{X}^{\bar{i}} \partial_{\bar{i}}, \quad \tilde{Y} = \tilde{Y}^i \partial_i +  \tilde{Y}^{\bar{i}} \partial_{\bar{i}}.
	\]
	
	By Lemma \ref{lem:Levi.Civita.connection}, the covariant derivative $\bar{\nabla}_{\tilde{X}} \tilde{Y}$ is given by
	\begin{equation} \label{eqn:connection.ideneity}
		\begin{split}
			\bar{\nabla}_{\tilde{X}} \tilde{Y} &= ( \tilde{X} \tilde{Y}^k + \tilde{X}^i \tilde{Y}^j \bar{\Gamma}_{ij} \mathstrut^{k} ) \partial_k + ( \tilde{X} \tilde{Y}^{\bar{k}} + \tilde{X}^{\bar{i}} \tilde{Y}^{\bar{j}} \bar{\Gamma}_{\bar{i}\bar{j}} \mathstrut^{\bar{k}} ) \partial_{\bar{k}}.
		\end{split}
	\end{equation}
	Evaluating at the point $x = (p, f(p))$ and using the primal coordinates $\xi$ on $G$, we have 
	\begin{equation} \label{eqn:covariant.derivative.computation}
		\begin{split}
			& \bar{\nabla}_{\tilde{X}} \tilde{Y}|_x = ( X Y^k + X^i Y^j \bar{\Gamma}_{ij} \mathstrut^{k} ) \partial_k + ( X Y^{\bar{k}} + X^{\bar{i}} Y^{\bar{j}} \bar{\Gamma}_{\bar{i}\bar{j}^*} \mathstrut^{\bar{k}} ) \partial_{\bar{k}} \\
			&\Rightarrow (\pi_0)_* (  \bar{\nabla}_{\tilde{X}} \tilde{Y} |_x ) = ( X Y^k + X^i Y^j \bar{\Gamma}_{ij} \mathstrut^{k} ) \partial_k.
		\end{split}
	\end{equation}
	By Lemma \ref{lem:info.geo.coefficients}, we have $\bar{\Gamma}_{ij} \mathstrut^{k}(\xi, \eta) = \Gamma_{ij}\mathstrut^{k}(\xi)$. So the last expression is equal to $\nabla_X Y$. Similarly, we have $\nabla^* = \bar{\nabla}^{\pi_1}$.
\end{proof}


Recall that $T_{(p, f(p))} G \subset T_{(p, f(p))} (M \times M') = T_p M \oplus T_{f(p)} M'$. Using this decomposition, define mappings $\iota_0, \iota_1: TG \rightarrow T(M \times M')$ as follows. If $v = v_0 \oplus v_1 \in T_{(p, f(p))} G \subset T_p M \oplus T_{f(p)} M'$, define $\iota_0(v), \iota_1(v) \in T_pM \oplus T_{f(p)} M'$ by
\begin{equation} \label{eqn:iota.maps}
	\iota_0(v) = v_0 \oplus 0, \quad \iota_1(v) = 0 \oplus v_1.
\end{equation}
In coordinates, if $v = a^i \partial_i + a^{\bar{i}} \partial_{\bar{i}}$, then $\iota_0(v) = a^i \partial_i + 0$ and $\iota_1(v) = 0 + a^{\bar{i}} \partial_{\bar{i}}$. Geometrically, $\iota_0(v)$ and $\iota_1(v)$ are respectively the horizontal and vertical components of $v$ in $T (M \times M')$. Now for vector fields $X, Y$ on $G$, we may rewrite the identify \eqref{eqn:connection.ideneity} in the form
\begin{equation} \label{eqn:connection.identity.intrinsic}
	\underbrace{\bar{\nabla}_X Y}_\text{$\in T_pM \oplus T_{f(p)} M'$} = \underbrace{\iota_0 (\nabla_X Y)}_\text{$\in T_p M \oplus 0$} + \underbrace{\iota_1(\nabla_X^* Y)}_\text{$\in 0 \oplus T_{f(p)} M'$}.   
\end{equation}

\begin{example}[Geometry of Bregman divergence] \label{eg:Bregman.geometry}
	As an illustration of the relation between the two geometries, let us consider the dualistic geometry of Bregman divergence. From Example \ref{eg:Bregman.divergence}, this corresponds to the case where $M = M' = \mathbb{R}^n$ and $c$ is the quadratic cost $c(p, q') = \frac{1}{2} |p - q'|^2$. In Euclidean coordinates, the matrix of the pseudo-Riemannian metric $h$, given by \eqref{eqn:metric.h.quadratic.case}, is constant. By Lemma \ref{lem:Levi.Civita.connection}, the Christoffel symbols of the Levi-Civita connection $\bar{\nabla}$ all vanish, so the $\bar{\nabla}$-geodesics  are constant-velocity straight lines in $\mathbb{R}^n \times \mathbb{R}^n$.
	
	For the quadratic transport, the graph $G$ has the form
	\[
	G = \{(p, \nabla \phi(p)) : p \in  \mathbb{R}^n \},
	\]
	where $\phi$ is a convex function. (The results still hold if $\phi$ is only defined on an open convex domain in $\mathbb{R}^n$.) So $q = \nabla \phi(p)$ is the dual coordinates obtained from the Legendre transformation \cite[Chapter 1]{A16} or equivalently the Brenier map. From \eqref{eqn:covariant.derivative.computation}, the Christoffel symbols of $\nabla$ (resp.~$\nabla^*$) in the primal (resp.~dual) coordinates vanish. So the primal (resp.~dual) geodesics on $G$ are straight lines in the primal (resp.~dual) coordinates. Thus we recover the classic dually flat geometry.   Also, from the first equation in \eqref{eqn:c.divergence.coefficients}, since $(c_{i:\bar{j}}) = -I$ and $\frac{\partial q}{\partial p} = D^2 \phi$, the Riemannian metric is given in primal coordinates by $(g_{ij}(p)) = D^2 \phi(p)$, the Hessian of $\phi$.
\end{example}

\subsection{Curvature tensors} \label{sec:curvature.tensors}
Next we study the Riemann curvature tensors $\bar{R}$ of $\bar{\nabla}$ on $M \times M'$, and $R, R^*$ of $\nabla$ and $\nabla^*$ respectively on $G$. To fix the notations, we define the Riemann curvature tensor (say for the primal connection $\nabla$) by
\[
R(X, Y)Z = \nabla_X \nabla_Y Z - \nabla_Y \nabla_X Z - \nabla_{[X, Y]} Z,
\]
where $[X, Y]$ is the Lie bracket. In coordinates, we write
\[
R_{ijk\ell} = g(R(\partial_i, \partial_j)\partial_k, \partial_{\ell})
\]
and ${R_{ijk}}^{\ell} = R_{ijkm} g^{m\ell}$, so that $R(\partial_i, \partial_j) \partial_k = {R_{ijk}}^{\ell} \partial_{\ell}$. We have 
\begin{equation} \label{eqn:RC.tensor.coeff}
	{R_{ijk}}^{\ell} = \partial_i {\Gamma_{jk}}^{\ell } - \partial_j {\Gamma_{ik}}^{\ell } + {\Gamma_{jk}}^m {\Gamma_{im}}^{\ell } - {\Gamma _{ik}}^m {\Gamma_{jm}}^{\ell }.
\end{equation}
See e.g.~\cite[Section 5.8]{A16}.
The notations for $R^*$ (on $G$) and $\bar{R}$ (on $M \times M'$) are analogous. Note that for $\bar{R}$ the indices run through both $\xi$ and $\eta'$.

\begin{lemma} \label{lem:R.bar.coeff}
	In the coordinates $(\xi, \eta')$, the coefficients of $\bar{R}$ are zero unless the number of unbarred and barred indices is equal, in which case the coefficient can be inferred from $R_{ij\bar{k}\bar{\ell}} = 0$ and
	\begin{equation} \label{eqn:R.bar.components}
		R_{i\bar{j}\bar{k}\ell}(\xi, \eta') = \frac{1}{2} \left( -c_{i\ell:\bar{j}\bar{k}} +  c_{i\ell:\bar{\beta}} c^{\bar{\beta}:\alpha} c_{\alpha:\bar{j}\bar{k}} \right),
	\end{equation}
	using the symmetries of the curvature tensor.
\end{lemma}
\begin{proof}
	This is a computation (done in \cite[Lemma 4.1]{KM10}) involving Lemma \ref{lem:Levi.Civita.connection}, Lemma \ref{lem:c.inverse.derivative} and \eqref{eqn:RC.tensor.coeff}, which is straightforward once one is familiar with the notations. Note that our expressions differ from \cite[(4.2)]{KM10} by a sign; this is due to the difference in the tensorial notation \eqref{eqn:RC.tensor.coeff}.
	
	For later use we also record the symmetries of the coefficients: 
	\begin{equation} \label{eqn:R.symmetries}
		\bar{R}_{i\bar{j}\bar{k}l} = -\bar{R}_{\bar{j}i\bar{k}\ell} = -\bar{R}_{i\bar{j}\ell\bar{k}} = \bar{R}_{\bar{k}\ell i \bar{j}}.
	\end{equation}
	See for example \cite[Proposition 7.4]{L97} whose notations are the same as ours. These symmetries hold in both the Riemannian and pseudo-Riemannian cases. Note that \eqref{eqn:R.symmetries} gives all the coefficients of $R$ that are possibly nonzero.
\end{proof}

\begin{lemma} \label{lem:curvature.coefficients} { \ }
	\begin{enumerate}
		\item[(i)] In primal coordinates, we have
		\begin{equation} \label{eqn:primal.curvature}
			R_{ijk\ell}(\xi) = -2 \bar{R}_{i\bar{\alpha}\bar{\beta}k}(\xi, \eta(\xi)) \frac{\partial \eta^{\bar{\alpha}}}{\partial \xi^j}  \frac{\partial \eta^{\bar{\beta}}}{\partial \xi^{\ell}} + 2 \bar{R}_{j\bar{\alpha}\bar{\beta}k}(\xi, \eta(\xi)) \frac{\partial \eta^{\bar{\alpha}}}{\partial \xi^i}  \frac{\partial \eta^{\bar{\beta}}}{\partial \xi^{\ell}}.
		\end{equation}
		\item[(ii)] In dual coordinates, we have
		\begin{equation} \label{eqn:dual.curvature}
			{R_{\bar{i}\bar{j}\bar{k}\bar{\ell}}}^*(\eta) = -2 \bar{R}_{\alpha\bar{i}\bar{k} \beta}(\xi(\eta), \eta) \frac{\partial \xi^{\alpha}}{\partial \eta^{\bar{j}}}  \frac{\partial \xi^{\beta}}{\partial \eta^{\bar{\ell}}} + 2 \bar{R}_{\alpha\bar{j}\bar{k} \beta}(\xi(\eta), \eta) \frac{\partial \xi^{\alpha}}{\partial \eta^{\bar{i}}}  \frac{\partial \xi^{\beta}}{\partial \eta^{\bar{\ell}}} .
		\end{equation}
	\end{enumerate}
\end{lemma}
\begin{proof}
	The proof is similar to that of \cite[Lemma 4.1]{KM10}. Now we use Lemma \ref{lem:info.geo.coefficients}, keeping in mind that in the primal case (say) $\eta = \eta(\xi)$ is a function of $\xi$. For example, we have
	\begin{equation*}
		\begin{split}
	\frac{\partial}{\partial \xi^i} {\Gamma_{jk}}^{\ell} &=
	 \left(c_{ijk:\bar{m}} + c_{jk: \bar{m}\bar{\alpha}} \frac{\partial \eta^{\bar{\alpha}}}{\partial \xi^i}\right) c^{\bar{m}:\ell} \\
	 &\quad - c_{jk:\bar{m}} \left( c^{\bar{m}:\alpha} c_{\alpha i : \bar{\beta}} c^{\bar{\beta}:\ell} + c^{\bar{m}: \alpha}c_{\alpha: \bar{\alpha} \bar{\beta}} c^{\bar{\beta}:\ell} \frac{\partial \eta^{\bar{\alpha}}}{\partial \xi^i} \right).
	\end{split}
\end{equation*}
	The dual case is similar.
\end{proof}

\begin{definition} [Unnormalized sectional curvature] \label{def:unnor.sec.cur}
	Let $X, Y$ be tangent vectors at the same point of $M \times M'$. We define the unnormalized sectional curvature of $\bar{R}$ by
	\begin{equation} \label{eqn:unnormalized.curvature}
		\overline{\sec}_u(X, Y) = h(\bar{R}(X, Y)Y, X).
	\end{equation}
	Similarly, we define
	\begin{equation} \label{eqn:unnormalized.curvature2}
		\sec_u(X, Y) = g(R(X, Y)Y, X), \quad \sec_u^*(X, Y) = g(R^*(X, Y)Y, X),
	\end{equation}
	when $X$ and $Y$ are tangent to $G$.
\end{definition}

\begin{remark} [The Ma-Trudinger-Wang tensor] \label{rmk:MTW}
	At a point $x = (p, q') \in M \times M'$, let $X = u \oplus 0$ and $Y = 0 \oplus \bar{v}$ where $u \in T_p M$ and $\bar{v} \in T_{q'} M$. Following Kim and McCann \cite{KM10}, the Ma-Trudinger-Wang (MTW) tensor (see \cite[Section 4]{GA14}) can be expressed intrinsically as the unnormalized (cross) sectional curvature
	\begin{equation} \label{eqn:MTW.tensor}
		\mathfrak{S} = \mathfrak{S}(u, \bar{v}) =  2 \overline{\sec}_u (u \oplus 0, 0 \oplus \bar{v}).
	\end{equation}
	(The constant $2$ comes from the $1/2$ in \eqref{eqn:R.bar.components}.) The cost $c$ is said to be a weakly regular cost if $\mathfrak{S} \geq 0$ whenever $u \oplus \bar{v}$ is a null tangent vector, i.e., $h(u \oplus \bar{v}, u \oplus \bar{v}) = 0$ \cite[Definition 2.3]{KM10}. Note that in this case we have $h(X, X) h(Y, Y) - h(X, Y)^2 = 0$; this is why one considers the unnormalized sectional curvature instead of the usual one. We refer the reader to \cite{V08, KM10, M12} and their references for how this condition comes into play in the regularity theory of optimal transport maps. In Corollary \ref{cor:MTW.interpret} we give a new information-geometric interpretation of this quantity.
\end{remark}

Now we are ready to state an interesting relation among the unnormalized sectional curvatures.

\begin{theorem} \label{thm:average.sec}
Let $X, Y \in T_{(p, f(p))}G \subset T_{(p, f(p))} M \times M'$. Then
\begin{equation} \label{eqn:curvatures.identity0}
\overline{\sec}_u(X, Y) = \frac{1}{2} \left( \sec_u(X, Y) + \sec_u^*(X, Y) \right).
\end{equation}
In particular, suppose the dualistic structure $(g, \nabla, \nabla^*)$ on $G$ has constant information-geometric sectional curvature $\lambda \in \mathbb{R}$. By definition, this means that
\begin{equation} \label{eqn:info.const.curvature}
\sec_u(X, Y) = \sec_u^*(X, Y) = \lambda (g(X, X) g(Y, Y) - g(X, Y)^2)
\end{equation}
for $X, Y$ tangent to $G$. Then 
\begin{equation} \label{eqn:constant.curvature}
\overline{\sec}_u(X, Y) = \lambda (h(X, X) h(Y, Y) - h(X, Y)^2) = \lambda (g(X, X) g(Y, Y) - g(X, Y)^2)
\end{equation}
	for $X, Y$ tangent to $G$.
\end{theorem}
\begin{proof}
	Write $X = x^i \partial_i + x^{\bar{i}} \partial_{\bar{i}}$ and $Y = y^i \partial_i + y^{\bar{i}} \partial_{\bar{i}}$. Using \eqref{eqn:unnormalized.curvature}, Lemma \ref{lem:R.bar.coeff} and the symmetries \eqref{eqn:R.symmetries} of $\bar{R}$, we have
	\begin{equation} \label{eqn:sec.bar.computation}
		\begin{split}
			\overline{\sec}_u(X, Y) &= \bar{R}_{i\bar{j}\bar{k}\ell} x^i y^{\bar{j}}y^{\bar{k}} x^{\ell} + \bar{R}_{k\bar{\ell}\bar{i}j} x^{\bar{i}} y^{j}y^{k} x^{\bar{\ell}} \\
			&\quad - \bar{R}_{j\bar{i}\bar{k}\ell} x^{\bar{i}} y^{j}y^{\bar{k}} x^{\ell} - \bar{R}_{i\bar{j}\bar{\ell}k} x^{i} y^{\bar{j}}y^{k} x^{\bar{\ell}}.
		\end{split}
	\end{equation}
	
	Since $X$ and $Y$ are tangent to $G$, from \eqref{eqn:tangent.vector.submanifold} we have $x^{\bar{i}} = x^i \frac{\partial \eta^{\bar{i}}}{\partial \xi^i}$ and $x^i = x^{\bar{i}} \frac{\partial \xi^i}{\partial \eta^{\bar{i}}}$ (similar for $Y$). In primal coordinate on $G$, we have $X = x^i \partial_i$ and $Y = y_i \partial_i$. We then compute
	\begin{equation} \label{eqn:sec.identity1}
		\begin{split}
			\sec_u(X, Y) &= R_{ijk\ell}x^i y^j y^k x^{\ell}\\
			&= 2\bar{R}_{j\bar{\alpha}\bar{\beta}k}  \frac{\partial \eta^{\bar{\alpha}}}{\partial \xi^i}\frac{\partial \eta^{\bar{\beta}}}{\partial \xi^{\ell}} x^i y^j y^k x^{\ell} - 2\bar{R}_{i\bar{\alpha}\bar{\beta}k}  \frac{\partial \eta^{\bar{\alpha}}}{\partial \xi^j}\frac{\partial \eta^{\bar{\beta}}}{\partial \xi^{\ell}} x^i y^j y^k x^{\ell} \\
			&= 2 \bar{R}_{j \bar{\alpha} \bar{\beta} k} x^{\bar{\alpha}} y^j y^k x^{\bar{\beta}} - 2 \bar{R}_{i \bar{\alpha} \bar{\beta} k} x^i y^{\bar{\alpha}} y^{k} x^{\bar{\beta}},
		\end{split}
	\end{equation}
	where the last identity follows from Lemma \ref{lem:curvature.coefficients}. Similarly, working in dual coordinates, we get
	\begin{equation} \label{eqn:sec.identity2}
		\sec_u^*(X, Y) = 2 \bar{R}_{\alpha \bar{j} \bar{k} \beta} x^{\alpha} y^{\bar{j}} y^{\bar{k}} x^{\beta} - 2 \bar{R}_{\alpha \bar{i}\bar{k} \beta} x^{\bar{i}} y^{\alpha} y^{\bar{k}} x^{\beta}.
	\end{equation}
	The result follows by averaging.
\end{proof}


\subsection{Divergence between geodesics}
Consider a Riemannian manifold with distance $d$. If $\gamma(s)$ and $\sigma(t)$ are two arc-length parameterized geodesics started at the same point when $s = t = 0$, then
\begin{equation} \label{eqn:sectional.curvature.classical.interpretation}
	\left. \frac{\partial^4}{\partial s^2 \partial t^2} d^2(\gamma(s), \sigma(t)) \right|_{s = t = 0} = - \frac{4}{3} \kappa \sin^2 \theta,
\end{equation}
where $\kappa$ is the sectional curvature of the the plane spanned by $\dot{\gamma}(0)$ and $\dot{\sigma}(0)$, and $\theta$ is the angle between the initial velocities (see for example \cite[(4.9)]{KM10}). This is the classical geometric interpretation of sectional curvature. In this section we extend this result to a $c$-divergence. Naturally, this involves the primal and dual geodesics rather than the Riemannian geodesics. The special case for $L^{(\alpha)}$-divergence is given in \cite{WY19}. This result (and its proof) is closely related to, but different from, \cite[Lemma 4.5]{KM10} which extends \eqref{eqn:sectional.curvature.classical.interpretation} to the pseudo-Riemanian framework with a general cost function. 

\begin{theorem} \label{thm:divergence.derivative.s2t2}
	Consider the graph $G$ and equip it with the dualistic structure $(g, \nabla, \nabla^*)$ induced by a $c$-divergence ${\bf D}$. Let $\gamma(s) = (p(s), f(p(s)))$ be a primal geodesic and $\sigma(t) = (f^{-1}(q'(t)), q'(t))$ be a dual geodesic with $\gamma(0) = \sigma(0) = x$. Letting
	\[
	X = \dot{p}(0) \oplus 0, \ Y = 0 \oplus \dot{q}'(0) \in T_{x} (M \times M'),
	\]
	we have
	\begin{equation} \label{eqn:c.divergence.forth.derivative}
		\begin{split}
			\left.  \frac{\partial^4}{\partial s^2 \partial t^2} {\bf D}[\gamma(s) : \sigma(t)] \right|_{s = t = 0} = \left. \frac{\partial^4}{\partial s^2 \partial t^2} c(p(s), q'(t)) \right|_{s = t = 0}= -2 \overline{\sec}_u (X, Y).
		\end{split}
	\end{equation}
\end{theorem}

Theorem \ref{thm:divergence.derivative.s2t2} may be regarded as an interpretation of the MTW tensor \eqref{eqn:MTW.tensor} on the graph $G$. It should be compared with the standard interpretation (see e.g.~\cite[(4.13)]{GA14}) which involves the $c$-exponential map.

\begin{corollary} [Information-geometric interpretation of the MTW tensor] \label{cor:MTW.interpret}
	In the context of Theorem \ref{thm:divergence.derivative.s2t2}, let $(p, q) \in G$, $u \in T_p M$ and, $\bar{v} \in T_q M$. Let $\gamma$ and $\sigma$ be respectively primal and dual geodesics whose initial velocities match with $u$ and $\bar{v}$ (when expressed in the respective coordinates). Then the MTW tensor can be expressed as
	\begin{equation} \label{eqn:MTW.interpret}
		\mathfrak{S}(u, \bar{v}) = -\left.  \frac{\partial^4}{\partial s^2 \partial t^2} {\bf D}[\gamma(s) : \sigma(t)] \right|_{s = t = 0}.
	\end{equation}
\end{corollary}

Although the statement of Theorem \ref{thm:divergence.derivative.s2t2} (as well as the proof) is very similar to \cite[Lemma 4.5]{KM10}, the two results are not the same. In \eqref{eqn:c.divergence.forth.derivative}, both $\gamma$ and $\sigma$ are curves in $G$. On the other hand, in \cite[Lemma 4.5]{KM10} one is a ``horizontal'' curve and the other one is ``vertical''. Before giving the proof of Theorem \ref{thm:divergence.derivative.s2t2} let us give some examples. Another application to the $L^{(\alpha)}$-divergence is given in Corollary \ref{cor:L.alpha.curvature.interpretation}.

\begin{example} [Bregman divergence]
	Consider Bregman divergence ${\bf D}$ as in \eqref{eg:Bregman.divergence}. Let $\xi$ and $\eta'$ be respectively the primal and dual coordinates. By Example \ref{eg:Bregman.geometry}, the primal and dual geodesics are given respectively by $\xi(s) = \xi(0) + \dot{\xi}(0)s$ and $\eta'(t) = \eta'(0) + \dot{\eta}'(0) t$. This gives
	\[
	\left.  \frac{\partial^4}{\partial s^2 \partial t^2} {\bf D}[\gamma(s) : \sigma(t)] \right|_{s = t = 0} = \left.  \frac{\partial^4}{\partial s^2 \partial t^2} \frac{1}{2}|\xi(s) - \eta'(t)|^2 \right|_{s = t = 0} =- \ddot{\xi}(0) \ddot{\eta}(0) = 0,
	\]
	which is consistent with the dual flatness. More about the constant curvature case is studied in Section \ref{sec:constant.curvature}.
\end{example}

\begin{example}[Quadratic cost on a Riemannian manifold] \label{ex:Loeper}
	Let $M$ be a Riemannian manifold with geodesic distance $d(p, q')$. Consider the cost function $c(p, q') = \frac{1}{2} d^2(p, q')$ on $M \times M$. Let $f = \mathrm{Id}$ be the identity transport, so that $G$ is the diagonal of $M \times M$. The corresponding $c$-divergence is ${\bf D} [ (p, p) : (p', p') ] = \frac{1}{2} d^2(p, p')$ (see the discussion before Example \ref{eg:Bregman.divergence}). Identifying $M$ and $G$ (as smooth manifolds) under the natural map $p \mapsto (p, p)$, it is easily shown that the dualistic structure $(g, \nabla, \nabla^*)$ reduces to the Riemannian structure of $M$, i.e., $g$ is the Riemannian metric of $M$ and $\nabla = \nabla^*$ are equal to the Riemannian Levi-Civita connection. In particular, the primal and dual geodesics are simply Riemannian geodesics. By  \eqref{eqn:sectional.curvature.classical.interpretation} and \eqref{eqn:MTW.interpret} we immediately get $\mathfrak{S} = \frac{2}{3} \kappa \sin^2 \theta$,
	where $\kappa$ is the Riemannian sectional curvature of the plane spanned by $(u, v)$ and $\cos \theta = g(u, v)$. This recovers \cite[Theorem 3.8]{L09}. Also see \cite[Example 3.6]{KM10}.
\end{example}

\begin{proof}[Proof of Theorem \ref{thm:divergence.derivative.s2t2}]
	As usual we use the primal coordinates for $\gamma$ and the dual coordinates for $\sigma$. The first equality follows directly from \eqref{eqn:c.divergence}. Computing the derivative \eqref{eqn:c.divergence.forth.derivative}, we have
	\begin{equation} \label{eqn:D.derivative.first.step}
		\frac{\partial^4}{\partial s^2 \partial t^2} {\bf D}[\gamma : \sigma]  = c_{i:\bar{k}} \ddot{\xi}^i \ddot{\eta}'^{\bar{k}} + c_{i:\bar{k}\bar{\ell}} \ddot{\xi}^i \dot{\eta}'^{\bar{k}} \dot{\eta}'^{\bar{\ell}} + c_{ij:\bar{k}} \dot{\xi}^i \dot{\xi}^j \ddot{\eta}'^{\bar{k}} + c_{ij:\bar{k}\bar{\ell}} \dot{\xi}^i \dot{\xi}^j \dot{\eta}'^{\bar{k}} \dot{\eta}'^{\bar{\ell}}.
	\end{equation}	
	By Lemma \ref{lem:info.geo.coefficients}, the primal and geodesic equations are given by
	\[
	\ddot{\xi}^k + c_{ij:\bar{m}}c^{\bar{m}:k} \dot{\xi}^i \dot{\xi}^j = 0, \quad
	\ddot{\eta}'^{\bar{k}} + c^{\bar{k}:m} c_{m:\bar{i}\bar{j}} \dot{\eta}'^{\bar{i}} \dot{\eta}'^{\bar{j}} = 0.
	\]
	Plugging into \eqref{eqn:D.derivative.first.step} and simplifying, we have
	\[
	\left.\frac{\partial^4}{\partial s^2 \partial t^2} {\bf D}[\gamma : \sigma] \right|_{(0, 0)} = (c_{ij:\bar{k}\bar{\ell}} - c_{ij:\bar{m}} c^{\bar{m}:m} c_{m:\bar{k}\bar{\ell}}) \dot{\xi}^i \dot{\xi}^j \dot{\eta}'^{\bar{k}} \dot{\eta}'^{\bar{\ell}} = -2R_{i\bar{k}\bar{\ell}j}  \dot{\xi}^i  \dot{\eta}'^{\bar{k}} \dot{\eta}'^{\bar{\ell}} \dot{\xi}^j.
	\]
	Comparing this with \eqref{eqn:unnormalized.curvature} gives the result.
\end{proof}


\section{Costs with constant sectional curvature}  \label{sec:constant.curvature}
The quadratic cost and Bregman divergence are flat when considered in both the pseudo-Riemannian and information geometric frameworks (see Example \ref{eg:quadratic.cost} and Example \ref{eg:Bregman.geometry}). In this section we consider the case of constant (non-zero) sectional curvature, a concept we now make precise. Note that the unnormalized sectional curvature $\overline{\sec}_u$ can be regarded as an operator on $\bigwedge^2 (T(M \times M')) = (\bigwedge^2 TM) \oplus (\bigwedge^2 TM') \oplus (TM \wedge TM')$ (see \cite[Remark 4.2]{KM10}). Since $\overline{\sec}_u$ vanishes on $(\bigwedge^2 TM) \oplus (\bigwedge^2 TM')$, $\overline{\sec}_u$ is determined by its action on $TM \wedge TM'$. 

\begin{definition} \label{def:const.curvature} 
	Consider a real-valued cost function $c$ on $M \times M'$.
	\begin{enumerate}
		\item[(i)] $c$ has constant cross curvature $\lambda \in \mathbb{R}$ on $TM \wedge TM'$ if
		\begin{equation} \label{eqn:const.nor.sect.cur}
			\begin{split}
				\overline{\sec}_u(X,Y) = \lambda \left(h(X,X)h(Y,Y) - h(X,Y)^2\right).
			\end{split}
		\end{equation}
		for any $X = v \oplus 0, Y = 0 \oplus \bar{v} \in T_{(p, q')} (M \times M')$.
		\item[(ii)] $c$ has constant sectional curvature $\lambda \in \mathbb{R}$ on a graph $G$ of optimal transport if \eqref{eqn:const.nor.sect.cur} holds when $X, Y$ are tangent to $G$. By Theorem \ref{thm:average.sec}, this is the case when $G$ has constant information geometric sectional curvature (see \eqref{eqn:info.const.curvature}).
		
	\end{enumerate}
\end{definition}

Note that when $X = v \oplus 0$ and $Y = 0 \oplus \bar{v}$, from the form of the metric $h$ (see \eqref{eqn:metric.h}) we always have $h(X, X) = h(Y, Y) = 0$. Thus in \eqref{eqn:const.nor.sect.cur} we have $\overline{\sec}_u(X, Y) = -\lambda h(X, Y)^2$.

\subsection{The logarithmic cost function}
Our main examples for Definition \ref{def:const.curvature} are the quadratic cost (Example \ref{eg:Bregman.divergence}) as well as  the logarithmic cost (Example \ref{ex:L.alpha}). As shown in \cite{PW16, W18}, the logarithmic cost arises naturally in stochastic portfolio theory \cite{PW16}. See \cite{PW18b} for probabilistic interpretations involving Dirichlet perturbations.

\begin{lemma} \label{lemma:log.geometry}
	Consider the logarithmic cost $c(p,q') = \frac{1}{\alpha} \log(1+\alpha p \cdot q')$ where $M = M' = (0, \infty)^n$ and $\alpha > 0$. Using the Euclidean coordinates (i.e., $\xi = p$ and $\eta' = q'$), we have the following coefficients for the induced pseudo-Riemannian geometry:
	\begin{equation} \label{eqn:log.metric.expr}
		c_{i:\bar{j}} = \frac{\delta_{i}^j}{1 + \alpha p \cdot q'} - \frac{\alpha p^{j} q'^{\bar{i}}}{(1 + \alpha p \cdot q')^2},
	\end{equation}
	\begin{equation} \label{eqn:log.connection.primal.expr}
		\bar{\Gamma}_{ij}\mathstrut^{k}  = c^{\bar{m}:k} c_{ij:\bar{m}} = -\frac{\alpha}{1 + \alpha p \cdot q'}\left( q'^{\bar{i}} \delta_{j}^k + q'^{\bar{j}} \delta_{i}^k\right),
	\end{equation}
	\begin{equation} \label{eqn:log.connection.dual.expr}
		\bar{\Gamma}_{\bar{i}\bar{j}}\mathstrut^{\bar{k}}  = c^{\bar{k}:m} c_{m:\bar{i}\bar{j}} = -\frac{\alpha}{1 + \alpha p \cdot q'}\left( p^i \delta_{j}^k + p^j \delta_{i}^k\right),
	\end{equation}
	\begin{equation} \label{eqn:log.cur.expr}
		\bar{R}_{i\bar{j}\bar{k}l} = \frac{\alpha}{2}  \left( c_{i:\bar{j}} c_{l:\bar{k}} + c_{i:\bar{k}}c_{l:\bar{j}}\right).
	\end{equation}
\end{lemma}

\begin{proof}
	The expressions \eqref{eqn:log.metric.expr}--\eqref{eqn:log.connection.dual.expr} can be obtained by direct computations. By Lemma \ref{lem:R.bar.coeff}, we have
	\begin{equation*}
		\begin{split}
			2\bar{R}_{i\bar{j}\bar{k}l} =& \frac{\alpha}{(1+\alpha p \cdot q')^2} \left(\delta_{i}^k\delta_{j}^l+\delta_{i}^j\delta_{k}^l\right) \\
			&- \frac{\alpha^2}{(1+\alpha p \cdot q')^3} \left(p^j q'^{\bar{l}}\delta_{i}^k + p^j q'^{\bar{i}}\delta_{k}^l + p^k q'^{\bar{l}}\delta_{i}^j + p^k q'^{\bar{i}}\delta_{j}^l\right) \\
			&+ \frac{2\alpha^3}{(1+\alpha p \cdot q')^4} \left(p^{\bar{i}} q'^j p^k q'^{\bar{l}}\right).
		\end{split}
	\end{equation*}
	We obtain \eqref{eqn:log.cur.expr} by comparing with \eqref{eqn:log.metric.expr}.
\end{proof}

In fact, as the following lemma shows, \eqref{eqn:log.cur.expr} is equivalent to the condition of constant cross curvature.

\begin{lemma} \label{lemma:const.sec.cur}
	A cost function $c$ has constant cross curvature $-4\alpha$ on $TM \wedge TM'$ if and only if \eqref{eqn:log.cur.expr} holds in some (and hence any) coordinate system $(\xi, \eta')$. In particular, the logarithmic cost function \eqref{eqn:alpha.cost} has constant cross curvature $-4\alpha$ on $TM \wedge TM'$.
\end{lemma}

\begin{proof}
	Fix a coordinate system $(\xi, \eta')$. Let $X, Y \in T_{(\xi, \eta')} (M \times M')$. Following the argument of \eqref{eqn:sec.bar.computation}, we have
	\begin{equation} \label{eqn:sec.repr.const}
		\overline{\sec}_{u}(X,Y) = \bar{R}_{i\bar{j}\bar{k}l} \left( x^iy^{\bar{j}}y^{\bar{k}}x^l + x^{\bar{j}}y^{i}x^{\bar{k}}y^l - x^iy^{\bar{j}}x^{\bar{k}}y^l - x^{\bar{j}}y^{i}y^{\bar{k}}x^l\right).
	\end{equation}
	
	Suppose $X  = (x^i \partial_i) \oplus 0$ and $Y  =  0 \oplus (y^{\bar{i}} \partial_{\bar{i}}) \in T_{(\xi, \eta')} (M \times M')$. Since $x^{\bar{i}} = 0$ and $y^i = 0$, \eqref{eqn:sec.repr.const} gives
	\begin{equation} \label{eqn:sec.repr.const.simple}
		\overline{\sec}_{u}(X,Y) = \bar{R}_{i\bar{j}\bar{k}l}x^iy^{\bar{j}}y^{\bar{k}}x^l.
	\end{equation}
	
	Suppose $c$ satisfies $\eqref{eqn:log.cur.expr}$. Then \eqref{eqn:sec.repr.const.simple} implies that
	\[
	\overline{\sec}_{u}(X,Y) = \frac{\alpha}{2} \left( c_{i\bar{j}} c_{l\bar{k}} + c_{i\bar{k}}c_{l\bar{j}}\right) x^i y^{\bar{j}} y^{\bar{k}} x^l = -4\alpha \cdot h^2(X,Y).
	\]
	Thus $c$ has constant cross curvature $-4\alpha$ on $TM \wedge TM'$ as $h(X,X) = 0 = h(Y,Y)$.
	
	Conversely, suppose $c$ has constant cross curvature $-4\alpha$ on $TM \wedge TM'$. Then
	\begin{equation} \label{eqn:sec.repr.const.simple.equi}
		\overline{\sec}_{u}(X,Y) = 4\alpha \cdot h^2(X,Y) = \frac{\alpha}{2} \left( c_{i\bar{j}} c_{l\bar{k}} + c_{i\bar{k}}c_{l\bar{j}}\right) x^i y^{\bar{j}} y^{\bar{k}} x^l.
	\end{equation}
	Note that both coefficients $\bar{R}_{i\bar{j}\bar{k}l}$ and $\frac{\alpha}{2} \left( c_{i\bar{j}} c_{l\bar{k}} + c_{i\bar{k}}c_{l\bar{j}}\right)$ are invariant under the swap of indices $i$ and $l$ or the swap of $\bar{j}$ and $\bar{k}$. Since \eqref{eqn:sec.repr.const.simple} and \eqref{eqn:sec.repr.const.simple.equi} holds for arbitrary choice of $x^i$, $x^l$ and $y^{\bar{j}}$, $y^{\bar{k}}$, the identity \eqref{eqn:log.cur.expr} must hold.
\end{proof}

\subsection{Consequences of constant cross curvature}
Now we show that if $c$ has constant cross curvature, then the statistical manifolds it generates have constant information-geometric sectional curvature.

\begin{theorem} \label{thm:const.cur}
	Suppose the cost function $c$ has constant cross curvature $-4\alpha$ on $TM \wedge TM'$. Then any graph $G$ of optimal transport has constant information-geometric sectional curvature $-\alpha$. Consequently, $c$ has constant sectional curvature $-\alpha$ on $G$.
\end{theorem}

\begin{proof}
	By Lemma \ref{lemma:const.sec.cur}, $c$ satisfies \eqref{eqn:log.cur.expr} in some coordinate system $(\xi, \eta')$. Let $G$ be a graph of optimal transport, and let $X$, $Y$ be tangent to $G$. By \eqref{eqn:sec.identity1}, we have
	\begin{align*}
		\sec_u(X, Y) &= \alpha \left( \left( c_{j\bar{\alpha}} c_{k\bar{\beta}} + c_{j\bar{\beta}}c_{k\bar{\alpha}}\right) x^{\bar{\alpha}} y^j y^k x^{\bar{\beta}} - \left( c_{i\bar{\alpha}} c_{k\bar{\beta}} + c_{i\bar{\beta}}c_{k\bar{\alpha}}\right) x^i y^{\bar{\alpha}} y^{k} x^{\bar{\beta}} \right) \\
		&= -\alpha \left( g(X,X)g(Y,Y) - g^2(X,Y)\right).
	\end{align*}
	Hence $G$ has constant information-geometric primal sectional curvature $-\alpha$. From \eqref{eqn:sec.identity2}, the same holds for the dual sectional curvature. The last statement follows from Theorem \ref{thm:average.sec}.
\end{proof}

\begin{remark}
	It is clear that if $c$ has constant cross curvature, then $c$ is weakly regular (see Remark \ref{rmk:MTW}). In fact, for $X \wedge Y \in TM \wedge TM'$, if $h(X,Y) = 0$ (i.e. $X \oplus Y$ is a null vector in the sense of \cite{KM10}) then the cross curvature of the cost function $c$ in Theorem \ref{thm:const.cur} vanishes everywhere. Thus the MTW tensor is identically zero. This recovers the recent result of Khan and Zhang (see \cite[p.22]{KZ19}). Also, Theorem \ref{thm:const.cur} recovers \cite[Theorem 18]{W18} which shows that the $L^{(\alpha)}$-divergence induces constant information geometric sectional curvature $-\alpha$.
\end{remark}

\begin{remark}
	It is interesting to know if some converse of Theorem \ref{thm:const.cur} holds: if all graphs of optimal transport have constant information geometric sectional curvature, does the corresponding cost also have constant cross curvature? Is the logarithmic cost (up to reparameterization and linear terms) the unique cost which has constant cross curvature?
\end{remark}

Now we specialize the above results to the $L^{(\alpha)}$-divergence to give an intrinsic interpretation of its information geometric sectional curvature. It is intrinsic because if a statistical manifold is dually projectively flat with constant sectional curvature $-\alpha$, then locally one can define canonically a divergence of $L^{(\alpha)}$-type which is consistent with the ambient geometry \cite[Theorem 19]{W18}.  See \cite{WY19} for more discussion and related results.

\begin{corollary} \label{cor:L.alpha.curvature.interpretation}
	Let ${\bf D}$ be the $L^{(\alpha)}$-divergence which is the $c$-divergence of the logarithmic cost \eqref{eqn:alpha.cost}. Consider the context of Theorem \ref{thm:divergence.derivative.s2t2}, so that $\gamma(s)$ is a primal geodesic, $\sigma(t)$ is a dual geodesic and $\gamma(0) = \sigma(0)$, we have
	\begin{equation} \label{eqn:curvature.interpretation.L.alpha}
		\left. \frac{\partial^2}{\partial s^2\partial t^2} {\bf D}[\gamma(s) : \sigma(t)] \right|_{s = t = 0} = -2 \alpha g^2(\dot{\gamma}(0), \dot{\sigma}(0)).
	\end{equation}
\end{corollary}
\begin{proof}
	Following the notations Theorem \ref{thm:divergence.derivative.s2t2}, Let $X = \dot{p}(0) \oplus 0$ and $0 \oplus \dot{q}'(0)$. By Theorem \ref{thm:divergence.derivative.s2t2}, we have
	\[
	\left. \frac{\partial^2}{\partial s^2\partial t^2} {\bf D}[\gamma(s) : \sigma(t)] \right|_{s = t = 0} = -2 \overline{\sec}_x(X, Y).
	\]
	Since the logarithmic cost has constant cross curvature $-4\alpha$ by Theorem \ref{thm:const.cur}, we have
	\begin{equation} \label{eqn:first.computation}
		\overline{\sec}_x(X, Y) = 4\alpha h^2(X, Y).
	\end{equation}
	
	Consider the coordinates $(\xi, \eta')$. Writing $X = u^i \partial_i \oplus 0$ and $Y = 0 \oplus \bar{v}^{\bar{j}} \partial_{\bar{j}}$. From \eqref{eqn:pseudo.Riemannian.metric}, we have
	\[
	h(X, Y) = \frac{-1}{2} c_{i:\bar{j}} u^i \bar{v}^{\bar{j}}.
	\]
	
	Now consider $\gamma$ and $\sigma$ as curves in $G$. Using the primal coordinate on $G$, we have
	\[
	\dot{\gamma}(0) = u^i \frac{\partial}{\partial \xi^i}, \quad \dot{\sigma}(0) = \bar{v}^{\bar{j}} \frac{\partial \xi^j}{\partial \eta^{\bar{j}}} \frac{\partial}{\partial \xi^j}.
	\]
	By Lemma \ref{lem:info.geo.coefficients}, we have
	\[
	g(\dot{\gamma}(0), \dot{\sigma}(0)) = -c_{i : \bar{m}} \frac{\partial \eta^{\bar{m}}}{\partial \xi^j} u^i \bar{v}^{\bar{j}} \frac{\partial \xi^j}{\partial \eta^{\bar{j}}} = -c_{i: \bar{j}} u^i \bar{v}^{\bar{j}}.
	\]
	Plugging this into \eqref{eqn:first.computation}, we obtain the desired result.
\end{proof}

\section{Conclusion} \label{sec:conclusion}
This paper uncovers a fundamental relation between optimal transport and information geometry, and we expect that this framework will be useful for extending results in information geometry using optimal transport and vice versa. Here we discuss several directions for further study.

In this paper we considered the basic Monge-Kantorovich optimal transport problem. We showed that a divergence can be associated to an optimal transport map (e.g.~the quadratic distance $\frac{1}{2}|x - y|^2$ corresponds to the identity transport for the quadratic transport problem). Thus, the choice of a divergence, or loss function, in a specific application  may be justifiable in terms of ideas from optimal transport. An interesting direction is to extend the pseudo-Riemannian framework and the results of this paper to the entropically relaxed transport problem
\begin{equation} \label{eqn:entropically.relaxed}
\mathcal{T}_{c, h}(\mu, \nu) := \inf_{\gamma \in \Pi(\mu, \nu)} \left( \int_{M \times M'} c d \gamma + h \mathrm{Ent}(\gamma) \right),
\end{equation}
where which $\mathrm{Ent}(\gamma)$ is the entropy of $\gamma$ and $h > 0$ (see Example \ref{eg:entropic} which interprets a modified Sinkhorn divergence as a $c$-divergence). The entropically relaxed transport problem is closely related to the Schr\"{o}dinger problem \cite{L13}. Note that the optimal coupling in \eqref{eqn:entropically.relaxed} is no longer concentrated on a graph, but as $h \rightarrow 0$ it converges to the optimal coupling in \eqref{eqn:MK.problem}. Thus, the corresponding information geometry, if any, lives on the product space $M \times M'$ rather than a submanifold $G$. This is close in spirit to \cite{AKO18} which considers the statistical manifold of optimal couplings. The recent paper \cite{P19} studies the limit of \eqref{eqn:entropically.relaxed} as $h \rightarrow 0$ and uses the $c$-divergence in a crucial way. It is natural to ask whether the quantities obtained in \cite{P19}, including a probabilistic approximation of the Schr\"{o}dinger bridge, can be understood geometrically. Another possible direction is to consider the geometry of dynamic optimal transport problems where the coupling is replaced by the law of a stochastic process, say $(X_t)_{0 \leq t \leq 1}$, with initial distribution $X_0 \sim \mu$ and final distribution $X_1 \sim \nu$. We believe that an improved understanding of these problems will be helpful in statistical applications of optimal transport and information geometry.

In \cite{WY19} and in Section \ref{sec:constant.curvature} we studied the geometric meaning of information-geometric curvature in the case of constant sectional curvature. It is desirable to extend this result to arbitrary statistical manifolds. While Theorem \ref{thm:divergence.derivative.s2t2} relates the time derivative of the divergence ${\bf D}[\gamma(s) : \sigma(t)]$ to the unnormalized sectional curvature $\overline{\mathrm{sec}}_u$, it is not intrinsic as there are infinitely many pseudo-Riemannian geometries which are compatible with a given dualistic structure $(G, g, \nabla, \nabla^*)$. The inverse problem (of constructing $h$, $c$ and ${\bf D}$) is related to the construction of {\it canonical divergence} \cite{AA15} in information geometry; see also \cite{FA18, FA18b, FA19}. Finally, let us remark that optimal transport problems such as the reflector antenna problem \cite{L11} and the $2$-Wasserstein transport on Riemannian manifolds may lead to interesting new examples of statistical manifolds and divergences.

\begin{acknowledgements}
Leonard Wong would like to thank Robert McCann, Jun Zhang and Soumik Pal for helpful conversations and comments. Jiaowen Yang would like to thank Shanghua Teng and Francis Bonahon for helpful discussions. Most of the work was done when he was a student at the University of Southern California. We also thank the anonymous reviewers for their careful reading and comments.
\end{acknowledgements}

%
\section*{Conflict of interest}
The authors declare that they have no conflict of interest.

\bibliographystyle{plain}      
\bibliography{geometry.ref}

\end{document}